\documentclass{amsart}
\usepackage{amssymb,amsfonts,amsmath}
\usepackage{color}

\copyrightinfo{2010}{American Mathematical Society}
\newtheorem{theorem}{Theorem}[section]
\newtheorem{lemma}[theorem]{Lemma}
\newtheorem{corollary}[theorem]{Corollary}

\theoremstyle{definition}
\newtheorem{definition}[theorem]{Definition}
\newtheorem{example}[theorem]{Example}
\newtheorem{problem}[theorem]{Problem}

\theoremstyle{remark}
\newtheorem{remark}[theorem]{Remark}
\numberwithin{equation}{section}

\begin{document}

\title[Extension problems]
      {Exterior extension problems for strongly elliptic operators: solvability and approximation using fundamental solutions} 

\author[V. Kalinin]{Vitaly Kalinin}
\address{EP Solutions SA, 
Avenue des Sciences 13, 1400 Yverdon-les-Bains, Switzerland}

\email{contact@ep-solutions.ch}
			
\author[A. Shlapunov]{Alexander Shlapunov}
\address{Siberian Federal University,
         Institute of Mathematics and Computer Science,
         pr. Svobodnyi 79,
         660041 Krasnoyarsk,
         Russia}
\email{ashlapunov@sfu-kras.ru}
				
\address{Sirius Mathematics Center
Sirius University of Science and Technology, Olimpiyskiy ave. b.1, 354349 Sochi,
         Russia}

\email{shlapunov.aa@talantiuspeh.ru}



\begin{abstract}
In this work we study three exterior extension problems for strongly elliptic partial 
equations: the Cauchy problem (in a special statement), the "analytical"{} continuation 
problem and the so called "inner"{} Dirichlet problem in the scale of the Sobolev spaces 
over a domain with relatively smooth boundaries. We consider the existence of solutions to 
these problems, the dense solvability and conditional well-posedness of these problems for a 
wide class of strongly elliptic systems. We also consider the approximation of solutions to 
these problems by a single layer potential and by a linear combination of "discrete"{} 
fundamental solutions in relation to a narrower class of strongly elliptic operators of the 
second order. The obtained results justify the applicability of the indirect method of 
boundary integral equations and  for numerical solving the 
exterior extension 
problems.
\end{abstract}

\maketitle

\section*{Introduction}
\label{S.0}

In this paper we consider three interconnected problems for elliptic partial differential 
equations: a special case of the Cauchy problem with a doubly connected boundary of the 
solution domain, the "analytical"{} continuation problem and a problem that can be named the 
"inner"{} Dirichlet problem. These problems can be classified as external extension problems.

Namely, let $\Omega_0$  and $\Omega_1$ be bounded domains in ${\mathbb R}^n$, $n \geq 2$,  
with sufficiently smooth boundaries $\partial \Omega_0$ and $\partial \Omega_1$, such that 
$\overline \Omega_0 \subset \Omega_1$ and let $L$ be an elliptic  
linear  partial differential operator of order $2m$, $m \in \mathbb N$, with real analytic coefficients 
in a neighbourhood of $\overline \Omega_1$. We also fix a boundary Dirichlet system 
$\{B_j\}_{j=0}^{2m-1}$ of order $(2m-1)$ on $\partial \Omega_0$, see Definition \ref{def.Dir} 
below; of course the standard system $\{ \frac{\partial^j}{\partial\nu^j}\}_{j=0}^{2m-1}$ of  
the normal derivatives with respect to $\partial \Omega_0$ belongs to this class.
Let us roughly formulate the problems, leaving the specification of function spaces 
for the next section.

The first one is the Cauchy problem for $L$.

\begin{problem} \label{eq:problem1}
Given data $u_0, \dots u_{2m-1}$ on $\partial \Omega_0$, 
find a solution $u$ to the homogeneous equation $Lu = 0$ in $\Omega_1 \setminus 
\overline \Omega_0$ satisfying in a suitable sense $B_j u = u_j$ on $\partial \Omega_0$, 
for all $0\leq j \leq 2m-1$.
\end{problem}  

The second one is the problem of the "analytic"{} continuation from a bounded domain to a 
large one.

\begin{problem} \label{eq:problem2}
Given data $V$ in $\Omega_0$ satisfying the %homogeneous 
equation $LV = 0$ in  $\Omega_0$ 
find a solution $U$ to  the homogeneous equation $LU = 0$ in  $\Omega_1$  
such that $U = V$ in  $\Omega_0$.
\end{problem}  

The third one is the so-called "inner Dirichlet problem"{}

\begin{problem} \label{eq:problem3}
Given data $v_0, \dots v_{m-1}$ on $\partial \Omega_0$ 
find a solution $U$ to  the homogeneous equation $LU = 0$ in  $\Omega_1$  
such that $B_j U = v_j$ on  $\partial \Omega_0$ for all $0\leq j \leq m-1$.
\end{problem}  

A typical example of Problem \ref{eq:problem1} is the 
Cauchy problem for the Laplace equation that were initially considered by Jacques Hadamard 
\cite{Had1923} as a famous example of an ill-posed problem. Since 
then the Cauchy problem for 
elliptic equations has been actively studying in various aspects, see, for instance, works by 
 S.N. Mergeljan \cite{Me56}, E.M. Landis \cite{Lan56} , M.M. Lavrent’ev \cite{Lv1}, 
R. Latt\'es and J.-L. Lions \cite{LatLi69}, 
L.E. Payne \cite{Pay1970},  V.A. Kozlov, V.G. Maz’ya and A.V. Fomin \cite{KMF91}, 
G. Alessandrini with co-authors \cite{AlRRV}  and others. 
 It appears that the regularization methods 
(see, for instance, \cite{TikhArsX}, \cite{LvSa}) are most effective for  
studying the problem.  The book \cite{Tark36} 
gives a rather full description of solvability conditions for the homogeneous 
elliptic equations in the Sobolev spaces. Actually, 
the Cauchy problem for elliptic equations and systems of equations arises in many fields 
of science and technology, including geophysics, elasticity theory (see, 
for instance, \cite{LvRShi}) and 
medical diagnostics (in particular, non-invasive electrocardiographic imaging \cite{Ramanathan2004}).

Problem  \ref{eq:problem2} initially arose as the analytical continuation problem in the 
theory of functions of a complex variable and it was subsequently generalized to harmonic 
functions and functions governed by a wide class of elliptic equations in ${\mathbb R}^n$. 
The long-time history of the study of various aspects of this problem includes, in 
particular, works of T. Carleman \cite{C33}, N. Aronszajn \cite{Ar56}, 
F.~E. Browder \cite{browder1962approximation}, B.M. Weinstock \cite{weinstock1973uniform}.  
A deep connection between Problems \ref{eq:problem1} and \ref{eq:problem2} was exploited in 
the paper \cite{ShTaLMS}, see also book by N. Tarkhanov \cite{Tark36}. 
The issues related to theoretical justifications of numerical solving of this problem have not been sufficiently studied. We can mention only one recent paper on quantifying the ill-conditioning of the classical 
analytic continuation problem, see L.N. Trefethen \cite{Tref2020}.

Problem \ref{eq:problem3} has an independent applied significance, since it arises in various fields of science and technology.
In particular, is related to non-contact cardiac electrical mapping with the basket intracardiac catheter \cite{grace2019high}. 

Moreover, Problems \ref{eq:problem1} and \ref{eq:problem2} can be reduced to Problem \ref{eq:problem3}.  This observation opens a way to construct numerical algorithms to obtain the Problem \ref{eq:problem1} solution and \ref{eq:problem2} via solving Problem \ref{eq:problem3}.

Besides that, Problem \ref{eq:problem3} underlies the following special "extension" approach to numerical solving classical boundary value problems for elliptic equations. In this paper, we will take the Dirichlet problem for the second order strongly elliptic operators as a model example of such a classical problem. The Dirichlet problem reads as follows:

\begin{problem} \label{eq:problem4}
Let $m=1$. Given data $w_0$ on $\partial \Omega_0$, 
find a solution $w$ to the %homogeneous equation 
$Lw = 0$ in $\Omega_0$ satisfying in a 
suitable sense $B_0 w = w_0$ on $\partial \Omega_0$.
\end{problem}  

The "extension"{} approach for solving Problem \ref{eq:problem4} consists in setting a 
"virtual"{} arbitrary boundary $\partial \Omega_1$ so that domain $\Omega_1$ bounded to it 
includes the closure of domain $\partial \Omega_0$: $\overline \Omega_0 
\subset \Omega_1$. A solution $w$ to Problem \ref{eq:problem4}  in $\Omega_0$  is approximated 
by restricting to $\Omega_0$ the  Problem \ref{eq:problem3} solution for the same elliptic operator $L$ and the same boundary datum $f_0=w_0$ on $\partial \Omega_0$ (i.e. $L u = 0 $ in  $ \Omega_1$, $ u=w_0 $ 
on $ \partial \Omega_0$).

Actually, the "extension"{} strategy is used relatively widely for numerical solving 
elliptic boundary value problems. For example, in relation to the finite element method the
"extension"{} approach is known as the fictitious domain method \cite{glowinski1994fictitious}
\cite{babuvska2005new},  \cite{agress2021smooth}.
The same approach is also used in the method of fundamental solutions, which will be 
discussed below.

The "extension"{} approach has some advantages. In particular, it allows a certain freedom 
of choice of geometry and smoothness of the prescribed embracing boundary 
$ \partial \Omega_1$. However, the well-posedness of boundary value Problem \ref{eq:problem4} 
does not imply the well-posedness for Problem \ref{eq:problem3}. Therefore, the applicability 
of the "extension"{} method needs a special justification which is provided in the presented paper.

The main objective of this paper is to consider a way to solve Problem \ref{eq:problem3} using the concept of fundamental solutions. 
We will consider this issue in several aspects: a) as an actual method for solving Problem \ref{eq:problem3} itself ; b) as a method for solving Problem \ref{eq:problem1} and Problem \ref{eq:problem2} which are reduced to Problem \ref{eq:problem3}; c) as a method for solving Problem \ref{eq:problem4} via the "extension" approach.

 We will focus on two methods of solving problem \ref{eq:problem3}. 
The first method is based on the representation of solution $u$ to Problem \ref{eq:problem3} in 
$\Omega_1$ by a potential of the single layer defined on boundary $ \partial \Omega_1$:
\begin{equation} \label{eq:sl}
u(x) = \int_{\partial \Omega_1} \Phi (x,y) v(y)d s_y, \, x \in \Omega_1, \, y 
\in \partial \Omega_1 ,
\end{equation}
where $\Phi (x,y)$ is the fundamental solution for operator $L$ and $v$ is the single layer 
density. 

The proposed method can be considered in the context of the classical boundary integral 
equation method, which is based on the representation of the boundary value problems 
solutions by the single layer or (and) double layer potential. Its numerical realization is 
known as the boundary element method (BEM) (see, for instance, \cite{rjasanow2007fast} or 
elsewhere). 

The single layer version of BEM is the simplest for numerical implementation, 
however, it requires some caution. In certain cases, despite the fact that a boundary value 
problem has a unique solution, the integral equation for the single layer density may be 
unsolvable (some examples are given in section \ref{S.3} of this paper). 
Moreover, the solvability of the equation depends on the specific type of fundamental solution.

The study of the solvability of the boundary integral equations has a long history, 
including influential works of S.G. Mikhlin \cite{Mikhl65},  A.P. Calder\'on and A. 
Zygmund (see, for instance,  \cite{Strain_1998}), 
M. Costabel \cite{Cost88}, M. Costabel  and  W. L. Wendland \cite{CoWe86}.
Note that the intergal equation method was mainly studied in relation to classical 
boundary value problems, such as the Dirichet and Neumann problems. Its application 
to Problem \ref{eq:problem3} needs to be considered separately.

The second method for solving Problem \ref{eq:problem3} taken into our consideration
is known as the method of fundamental solutions (MFS).
It consists in approximating the solution of an elliptic equation in a bounded domain by a 
linear combination of fundamental solutions, whose singularities belongs to a discreet set of 
the isolated points of the prescribed boundary embracing the domain. More precisely, the solution  $u$ to  
Problem \ref{eq:problem3} in $\Omega_1$ is approximated as follows:

\begin{equation} 
\label{eq:mfs}
u \approx 
{u}_N(x) = \sum_{j=1}^{N}  %\cdot 
\Phi(x,y_j) v_j,\ x \in \Omega,\ y_j \in \{y_j\}_{j=1}^{N} 
\subset \partial \Theta_{ex} 
\end{equation}
where $v_j$ are real weight vector coefficients, $\partial \Theta_{ex}$ 
 is the "embracing"{} boundary, i.e., it is a boundary of a domain $\Theta_{ex}$  such that
$\overline \Omega_1\subset \Theta_{ex}$.

MFS can be viewed as a discrete version of the single layer potential method 
\cite{fairweather1998method}. MFS has a number of attractive properties (see the discussion 
of its advantages in \cite{smyrlis2009applicability}). For instance, it does not require 
numerical integration. Therefore, MFS is widely used for numerical  solving of classical boundary 
value problems for partial differential equations arising in various fields 
(see \cite{cheng2020overview} and the literature cited there). 
To justify the applicability of MFS to solving problems for an elliptic equation in a given 
domain, it is necessary to show the existence of real coefficients $\{v_j\}_{j=1}^{N}$, 
$j,N \in \mathbb N$, so that red (finite or infinite) sums of type 
\eqref{eq:mfs}
are dense in a space of solutions to the equation in the solutions domain.
This fact has been shown for some strong elliptic equations including the Laplace, Helmholtz, 
biharmonic and polyharmonic equation (using various functional spaces) subject to some 
smoothness requirements of the boundary geometry and boundary condition (see \cite{Bogo85}, 
\cite{smyrlis2009density}, \cite{smyrlis2009applicability} , \cite{smyrlis2011} and the 
literature cited in \cite{cheng2020overview}). 
%\textcolor{red}{For general elliptic 
%systems admitting fundamental solutions see \cite[Theorem 4.2.1]{Tark97}.} 
The question of further generalization of these results remains open.

Some other results adjoin this topic. The completeness of the set of the discrete fundamental 
solutions with a different localization of their singularities, namely: $\Phi(x_i,y)$, $x_i \in \partial \Omega_0$, 
$y \in  \Omega_2 \setminus \overline \Omega_1$, 
$\overline \Omega_0\subset \Omega_1$, $\overline \Omega_1\subset \Omega_2$
were studied by A.B. Bakushinsky \cite{bakushinskii1970remarks}, 
M.A. Alexidze \cite{aleksidze1991fundamental}, cf. also the pioneer  result 
\cite{R1885} by C. Runge for holomorphic functions where the Cauchy kernel was used,  
or %\cite{Tark97} 
 \cite[Theorem 4.2.1]{Tark97}
 for elliptic differential operators admitting left fundamental 
solutions. It was also shown that the linear combination of fundamental solutions with 
singularities in an arbitrary open set $U$ outside the closure of a connected domain 
$\Omega$ are dense (in the sense of uniform norm) in the space 
$X=\{u\in C^m (\Omega): \, Lu =0 \mbox{ in } \Omega \}\cap C (\overline \Omega)$,  
see \cite{browder1962approximation}. This result has been significantly expanded in works 
\cite{weinstock1973uniform}, \cite{Tark36}, \cite{smyrlis2011}.  However, when singularities 
of the fundamental solution belong only boundary $\partial \Omega$ 
of domain $\Omega$, this statement not always valid \cite{smyrlis2009density}.

In addition, we note that the single layer integral equation method (see\cite{LvSa}, \cite{LiuZh2013}, \cite{BorCh2016})  and MFS (see  \cite{wang2006application} and \cite{ZhWe2008}) have been used for numerical solving the Causchy problem for the Laplace equation 
 in a domain with the doubly-connected boundary (similar to Problem \ref{eq:problem1}).  In these works the single layer density was given on both boundaries $ \partial \Omega_0$ and $ \partial \Omega_1$ and the 
fundamental solutions singularities were given on the two  closed surfaces 
$\partial \Theta_{in}$ and $\partial \Theta_{ex}$ such that $\overline \Theta_{in} \subset 
\Omega_0$, $\overline \Omega_1  \subset \Theta_{ex}$.
In this paper, we do not consider the justification of these approaches focusing on the proposed method of reduction of the Cauchy problem to Problem \ref{eq:problem3} which assumes the single layer density or the fundamental solution singularities  to be given only on one surface.

The paper is organized as follows. Section \ref{s.math.prelim} is devoted to 
mathematical preliminaries such as the notations and the used function spaces. 
In section \ref{S.2}, we present some results on solvability and conditional 
well-posedness of extension Problems \ref{eq:problem1}, \ref{eq:problem2}, \ref{eq:problem3} 
for strongly elliptic systems in the form of the generalized Laplacian in bounded domains 
in ${\mathbb R}^n$, $n\geq 2$, with relatively smooth boundaries, considering 
their solutions in Sobolev-Slobodetsky spaces. In section \ref{S.3} we investigate the 
representation and approximation of solutions to the strongly elliptic systems of the second 
order by the single layer potential and a weighted sum of the fundamental solutions. 
In section \ref{S.4} we provide some results justifying the applicability of the single layer and MFS 
methods for approximation of solutions to Problems \ref{eq:problem1}--\ref{eq:problem4} 
for the strongly elliptic systems of the second order.

\section{Mathematical preliminaries}
\label{s.math.prelim}

Let $\theta$ be a measurable set in ${\mathbb R}^n$, $n\geq 2$.
Denote by $L^2(\theta)$ a Lebesgue space of functions on  $\theta$ with the 
standard inner product $\left( u,v \right)_{L^2(\theta)}$. 
%\begin{equation*}
%\left( u,v \right)_{L^2(\theta)} = \int_{\theta} v (x)u (x)\ dx .
%\end{equation*}
If $D$ is a domain in ${\mathbb R}^n$ with a Lipschitz  
%piecewise smooth 
boundary $\partial D$, then for $s \in \mathbb{N}$ we denote by $H^s(D)$ 
the standard Sobolev space with the standard inner product $\left( u,v \right)_{H^s(D)} $, 
see, for instance, \cite{Ad03}. 
%\begin{equation*}
%\left( u,v \right)_{H^s(D)} = \int_{D}\sum_{|\alpha| \le s} 
%(\partial^{\alpha}v) (\partial^{\alpha}u)
%dx .
%\end{equation*}
It is well-known that this scale extends for all real $s>0$. More precisely, 
given any non-integer $s > 0$, we use  
the so-called Sobolev-Slobodetskii space $H^s (D)$,  
%isdefined to be the completion of $C^{\infty} (\overline{D})$ with respect to the norm
%$$
%   \| u \|_{H^s (D)}
% = \Big( \| u \|^2_{H^{[s]} (D)}
%       + \int \!\!\! \int_{D \times D}
%         \sum_{|\alpha | = [s]}
%         \frac{|\partial^\alpha u (x) - \partial^\alpha u (y)|^2}{|x - y|^{n+2(s-[s])}}\,
%         dx dy
%   \Big)^{1/2},
%$$
%where $[s]$ is the integer part of $s$, 
see \cite{Slob58}. We denote also by $H^s_0 (D)$ the closure of the subspace 
$C^{\infty}_{\mathrm{comp}} (D)$ in $H^{s} (D)$, where $C^{\infty}_{\mathrm{comp}} (D)$ is 
the linear space of functions with compact supports in  $D$. 
Then the scale of Sobolev spaces can be extended for negative smoothness indexes, too. 
Namely, for $s > 0$, the space $H^{-s} (D)$
%denote by $H^{-s} (D)$ the completion of
%$C^{\infty} (\overline{D})$ with respect to the norm
%$$
%   \| u \|_{H^{-s} (D)}
% = \sup_{\substack{v \in C^{\infty}_{\mathrm{comp}} (D) \\ v \ne 0}}
%   \frac{|(v,u)_{L^2 (D)}|}{\| v \|_{H^{s} (D)}}.
%$$
%Actually, $H^{-s} (D)$ 
can be identified with the dual of $H^s_0 (D)$ with respect to the pairing induced by 
$(\cdot, \cdot)_{L^2 (D)}$, see, for instance, \cite{Ad03}.

If the boundary $\partial D$ of the domain $D$ is sufficiently smooth, 
then, using the standard volume form $d\sigma$ on the hypersurface $\partial D$ induced from 
${\mathbb R}^n$, we may consider the Sobolev-Slobodeckij spaces  $H^s(\partial D)$ on 
$\partial D$, see, for instance, \cite{Ad03}. 
%Namely, let $L^2(\partial D)$ be the Lebesgue space of 
%functions on  $\partial D$ with the inner product
%\begin{equation*}
%\left( u,v \right)_{L^2(\partial D)} = \int_{\partial D} v (x)u (x)\ d\sigma (x) .
%\end{equation*}
%If $0<s<1$ and $\partial D\in C^1$ then we define $H^s(\partial D)$
%to be the completion of $C^{1} (\partial D)$ with respect to the norm
%$$
%   \| u \|_{H^s (\partial D)}
% = \Big( \| u \|^2_{L^{2} (\partial D)}
%       + \int \!\!\! \int_{\partial D \times \partial D}
%         \frac{|u (x) - u (y)|^2}{|x - y|^{n-1+2s}}\,
%         d\sigma (x) d\sigma (y)
%   \Big)^{1/2}.
%$$
%For $s\geq 1$ we have to consider more smooth hypersurfaces. 
%For instance, if $\partial D \in C^{[s]+1}$) then we may define the space $H^s (\partial D)$
%using local coordinates on $\partial D$ and a suitable partition of unity. 

%In this section we recall both classical and recent results related to 
%elliptic  and parabolic differential  operators. With this purpose, 
Recall that a 
linear (matrix) differential operator 
$$
A (x,\partial) = \sum_{|\alpha|\leq m} A_\alpha (x) \partial^\alpha
$$  
 of order $m$ and with $(l\times k)$-matrices $A_\alpha (x) $ having 
entries from $C^\infty (X)$ on an open set $X \subset {\mathbb R}^n$, 
is called an operator  with injective (principal) symbol on $X$ 
if $l\geq k$ and for its principal symbol
$$
\sigma(A) (x,\zeta) = \sum_{|\alpha|= m} A_\alpha (x) \zeta^\alpha
$$ 
we have 
$
\mathrm{rang} \, ( \sigma(A) (x,\zeta) )=k \mbox{ for any } x\in X , \zeta \in {\mathbb R}^n 
\setminus \{0\} 
$. 
An operator $A$  is called (Petrovsky) elliptic, if $l=k$ and 
its symbol is injective.

An operator $L (x,\partial)$ is called strongly elliptic if 
it is elliptic, its order $2m$ is even and there is a positive constant $c_0$ such that 
\begin{equation} \label{eq.str.ell}
(-1)^{m}\Re{\, (w^*\sigma(L)} (x,\zeta) \,w)\geq c_0 |\zeta|^{2m} |w|^2
\mbox{ for any } x\in X , \zeta \in {\mathbb R}^n , 
w \in {\mathbb C}^k
\end{equation}
where $w^* = \overline w^T$ and $w ^T$ is the transposed vector 
for $w \in {\mathbb C}^k$.

In fact, \eqref{eq.str.ell} yields that the so-called G\aa{}rding inequality: 
\begin{equation} \label{eq.Garding}
\|u\|^2_{[H^m (X)]^k}  \leq c_1 \Re{((Lu,u)_{[L^2 (X)]^k}}) + 
c_2 \|u\|^2_{[L^2 (X)]^k} 
\mbox{ for all } u \in [H^m_0 (X)]^k
\end{equation}
with some positive constants $c_1$, $c_2$ independent on $u$.

Actually, if the principal symbol of $A$ is injective then  
the operator $A^*A$ is strongly elliptic of order $2m$ where
$$
A^* (x,\partial) = \sum_{|\alpha|\leq m} (-1)^{|\alpha|} \partial^\alpha (A^*_\alpha (x) \cdot)
$$
is the formal adjoint for $A$ with the adjoint matrices $A^*_\alpha (x) $. 
The typical operator of such type is the Laplacian 
$
-\Delta = \nabla^*\nabla = - \nabla \cdot \nabla.
$ 
For the generalized Laplacians $A^*A$ {G\aa{}rding inequality \eqref{eq.Garding} 
is equivalent to the following:
\begin{equation} \label{eq.Garding.Laplace}
\|u\|^2_{[H^m (X)]^k}  \leq c_1 
\|A u \|^2_{[L^2 (X)]^k} + 
c_2 \|u\|^2_{[L^2 (X)]^k} 
\mbox{ for all } u \in [H^m_0 (X)]^k
\end{equation}
with some positive constants $c_1$, $c_2$ independent on $u$.

Next, for a domain $D\subset X$ denote by $S_A (D)$ the space of solutions 
to the equation 
$
A u = 0 \mbox{ in } D
$ 
in the sense of distributions. If the principal symbol of $A$ is injective
then elements of $S_A (D)$ 
are actually $C^\infty$-smooth (vector-)functions in $D$.

We say that a differential operator $Q$ satisfies 
the so called \textit{Uniqueness condition in the small} on $X$ or 
\textit{the Unique Continuation Property} if
\begin{itemize}
\item[(US)]
for any domain $D\subset X$ if a solution $u \in S_{Q} (D)$ vanishes 
on an open subset $G\subset D$ then it is identically zero in $D$.
\end{itemize}
In the sequel we assume that the both operators $L$ and $L^*$ under the consideration 
possess the Unique Continuation Property (US)  (in particular, we assume 
that the generalized Laplacian $A^*A$ possesses this 
property).  %As the elliptic operator $A^*A$ is formally self-adjoint, 
This assumption provides that the operator admits a bilateral 
fundamental solution, say $\Phi$,  see \cite[\S 2.3]{Tark95a}. As it is known 
the property (US) holds true for any elliptic operator 
with real analytic coefficients. Also scalar second order operators 
with $C^1$ smooth coefficients  possess the Unique Continuation Property, 
see \cite{Ar56}, \cite{Cor56}, \cite{Lan56}. Actually, for  
general elliptic operators with smooth coefficients the absence of  
multiple complex characteristics is essential for the property (US).

However both the property (US) and 
the real analyticity of the coefficients are too excessive for the existence 
of the fundamental solution; for instance, a bilateral fundamental solution 
exists on a domain $D\subset X$ for the operator $A^*A$ if the operator $A$ has injective 
principal symbol and satisfies (US) on $X$. Actually, in this case the Green function 
of the Dirichlet problem for the Laplacian $A^*A$ in $D$ is a bilateral 
fundamental solution for it. 

More precisely, we begin with a definition.

\begin{definition} \label{def.Dir}
A set of linear differential 
operators $\{B_0,B_1, \dots B_{m-1}\}$ is called a 
$(k\times k)$-Dirichlet system of order $(m-1)$ 
on $\partial D$ if: 
1) the operators are defined in a neighbourhood of $\partial D$; 
2) the order of the differential operator $B_j $ equals to $j$; 
3) the map $ \sigma (B_j) (x,\nu (x)) :{\mathbb C}^k \to {\mathbb C}^k$ 
is bijective for each $x \in \partial D$, where 
$\nu (x)$ will denote the outward normal vector to the hypersurface $\partial D$
at the point $x\in \partial D$.
\end{definition}
A typical $(k\times k) $-Dirichlet system of order $(m-1)$ on $\partial D$
consists of 
$$
\{I_k, I_k \frac{\partial}{\partial\nu}, \dots, 
I_k \frac{\partial^{m-1}}{\partial\nu^{m-1}} \}
$$ 
where 
$I_k$ is the unit $(k\times k)$-matrix and $\frac{\partial^j}{\partial\nu^j} $ 
is $j$-th normal derivative with respect to $\partial D$. 
%It is known that the functions of $H^s (D)$, where $s > 1/2$, possess well-defined traces on 
%the Lipschitz surface $\partial D$. 
%For $s\in \mathbb N$, the trace operator 
%\begin{equation} \label{eq.trace}
%t_s: H^s (D) \to H^{s-1/2} (\partial D)
%\end{equation}
%obtained in this way acts continuously, if $\partial D\in C^s$;
%moreover, in this case it possesses a bounded right inverse, see for instance 
%\cite[Ch.~1, \S~8]{LiMa72} and \cite{McL00}. 
According to the Trace Theorem, see for instance 
\cite[Ch.~1, \S~8]{LiMa72} and \cite{McL00}, 
if $\partial D\in C^{s}$, $s\geq m\geq 1$ then 
%for each $s \in \mathbb N$, $s\geq 2$, 
each operator $B_j$ induces a bounded linear operator
\begin{equation} \label{eq.trace.B_j}
B_j: H^s (D) \to H^{s-j/2} (\partial D).
\end{equation}

The main advantage of the use of the Dirichlet system is the following lemma.

\begin{lemma} \label{l.Dir.right}
Let $\partial D\in C^s$, $s\geq m$ and 
$B=\{B_0,B_1,\dots, B_{m-1}\}$ be a Dirichlet system of order $(m-1)$ 
on $\partial D$. Then for each set $\oplus_{j=0}^{m-1} 
u_j \in \oplus_{j=0}^{m-1} [H^{s-j-1/2} (\partial D)]^k$
there is a function $u \in  [H^s (D)]^k$ such that 
\begin{equation} \label{eq.Dir.right}
\oplus_{j=0}^{m-1} B_j u =\oplus_{j=0}^{m-1} u_j \mbox{ on }\partial D.
\end{equation}
%Such a function is unique, if it is  orthogonal to $[H^s (D) \cap H^m_0 (D)]^k $.
\end{lemma}

\begin{proof} See, for instance, \cite[Lemma 5.1.1]{Roit96}.
%For the existence see, for instance, \cite[Lemma 5.1.1]{Roit96}.
%Of course, there are infinitely many solutions to \eqref{eq.Dir.right}.
%However, according to \cite{HedbWolf1} the null-space of the operator 
%\begin{equation} \label{eq.trace.B}
%(\oplus_{j=0}^{m-1}B_j)_s: [H^s (D)]^k \to \oplus_{j=0}^{m-1} [H^{s-j/2} (\partial D)]^k.
%\end{equation}
%coincides with $[H^s (D) \cap H^m _0 (D)]^k$. As $[H^s (D) \cap H^m _0 (D)]^k$ is a closed 
%subspace in the space $H^s (D)$, using Direct Sum Theorem  we see that 
%$$
%[H^s (D)]^k = (H^s (D) \cap H^m _0 (D)) \oplus ([H^s (D) \cap H^m _0 (D)]^k) ^{\bot}
%$$
%where $\Sigma^\bot$ is the orthogonal complement of a set $\Sigma\subset [H^s (D)]^k $.
%Imposing the orthogonality conditions providing the uniqueness of the solution, 
%one may define the so-called right inverse operator 
%\begin{equation} \label{eq.Dir.right.inv}
%\Big(\oplus_{j=0}^{m-1} B_j\Big)^{-1}_{r,\bot}: 
%\oplus_{j=0}^{m-1} [H^{s-j-1/2} (\partial D)]k \to ([H^s (D) \cap H^m _0 (D)]^k) ^{\bot} 
%\subset H^s (D).
%\end{equation}
%Note that one may define the right inverse operator in different ways!
\end{proof}

Next we need the 
following useful (first) Green formula. % \eqref{eq.Green.M.one}.

\begin{lemma} \label{eq.dual.Dir}
Let $m\in \mathbb N$, 
$\partial D\in C^m$,  $A$ be an differential operator 
with injective symbol of order $m \in \mathbb N$ in a neighbourhood of 
$\overline D$ and $B=\{B_0,B_1, \dots B_{m-1}\}$ be a Dirichlet system of order $(m-1)$ on $
\partial D$. Then there is a  Dirichlet system ${\tilde B }^A=\{\tilde B_0^A, \tilde 
B_1 ^A,\dots 
\tilde B_{m-1}^A \}$ on $\partial D$ such that 
for all $v \in [H^{m} (D)]^k$, $ u\in  [H^m (D)]^k$ we have  
\begin{equation} \label{eq.Green.M.B}
\int_{\partial D} \Big( \sum_{j=0}^{m-1}({\tilde B}^A_{m-1-j} v)^* B_j u 
\Big) d\sigma = \int_{D} \Big( v^*  A u - (A^* v)^{*} u  \Big) dx.
\end{equation}
\end{lemma}

\begin{proof} See, for instance, \cite[Lemma 8.3.3]{Tark97}.
\end{proof}

Now we recall the Existence and Uniqueness 
Theorem for the Dirichlet Problem related to strongly elliptic operators.

\begin{problem} \label{pr.Dir}
 Given pair 
$g \in [H^{s-2m} (D)]^k$ and $\oplus_{j=0}^{m-1} w_j \in 
\oplus_{j=0}^{m-1}  [H^{s-j-1/2} (\partial D)]^k$ find, if possible a 
function $w \in [H^{s} (D)]^k$ such that 
\begin{equation} \label{eq.Dirichlet.Laplacian}
\left\{ \begin{array}{lll}
L  w =g & {\rm in} & D,\\
\oplus_{j=0}^{m-1} B_j w= \oplus_{j=0}^{m-1} w_j & {\rm on} & \partial D.\\
\end{array}
\right.
\end{equation}
\end{problem}
The problem can be treated in the framework of the operator theory in Banach spaces, regarding 
\eqref{eq.Dirichlet.Laplacian} as an operator equation 
with the linear bounded operator 
$$
(L, \oplus_{j=0}^{m-1} B_j) : [H^{s} (D)]^k \to [H^{s-2m} (D)]^k \times 
\oplus_{j=0}^{m-1}  [H^{s-j-1/2} (\partial D)]^k, \, s\geq m. 
$$
Recall that a problem related to operator equation 
$
R u =f
$ 
with a linear bounded operator $R: X_1 \to X_2$ in Banach spaces $X_1, X_2$ has the Fredholm 
property, if the kernel ${\rm ker}(R)$ of the operator $R$ and 
the co-kernel ${\rm coker}(R)$ (i.e. 
the kernel ${\rm ker}(R^*)$ of its adjoint operator $R^*: X_2^* \to X_1^*$)
are finite-dimensional vector spaces and the range of the operator $R$ is closed in $X_2$.  

\begin{theorem} \label{t.Dirichlet.M}
Let $L$ be a strongly elliptic differential operator of order $2m$, $m\geq 1$, 
with smooth coefficients  in a neighbourhood $X$ of $\overline D$,  
$\partial D\in C^s$, $s\geq m$ and 
$B=\{B_0,B_1,\dots, B_{m-1}\}$ be a Dirichlet system of order $(m-1)$ 
on $\partial D$. Then Problem \ref{pr.Dir} has the Fredholm property on the scale of the Sobolev spaces over $D$.  Namely, the dimensions of the spaces 
$[H^{s} (D) \cap H^{m}_0 (D)]^k\cap S_L (D)$ and 
$[H^{s} (D) \cap H^{m}_0 (D)]^k\cap S_{L^*} (D)$ 
are finite and  there is a bounded linear operator 
$$
{\mathcal G}_D: [H^{s-2m} (D)]^k \to [H^{s} (D)]^k \cap [H^{m}_0 (D)]^k, 
$$
$$
\Pi^{(1)}_D: [H^{s} (D)]^k \cap [H^{m}_0 (D)]^k \to [H^{s} (D)]^k \cap [H^{m}_0 (D)]^k
\cap S_L (D) , 
$$
$$
\Pi^{(2)}_D: [H^{s-2m} (D)]^k  \to [H^{s} (D)]^k \cap [H^{m}_0 (D)]^k
\cap S_{L^*} (D) , 
$$
such that 
\begin{equation} \label{eq.Hodge}
{\mathcal G}_D L = I- \Pi^{(1)}_D 
\mbox{ on } [H^s (D) \cap H^{m}_0 (D)]^k, \,\,L {\mathcal G}_D = 
I- \Pi^{(2)}_D \mbox{ on } [H^{s-2m} (D) ]^k.
\end{equation}
\end{theorem}

\begin{proof} For the Fredholm property see, for instance, \cite{Mikh76}, \cite[Ch. 5]{Roit96} or elsewhere.  As for Hodge decomposition \eqref{eq.Hodge}, one may  
refer, for instance, to \cite{Mikh76}, \cite{GiTru83} for the second order elliptic 
operators or \cite{Brow59a} or any modern book devoted to elliptic
partial differential equation. For the Hodge decomposition related to the Dirichlet 
problem for the generalized Laplacian $A^*A$ 
in the Sobolev spaces of negative smoothness see \cite[Ch. 5]{Roit96} or 
\cite{ShTaDual}. 
\end{proof}

The operator ${\mathcal G}_D$ is called usually the Hodge 
parametrix of Dirichlet Problem \ref{pr.Dir}; if $\Pi^{(1)}_D  = \Pi^{(2)}_D =0$ then ${
\mathcal G}_D$ is called Green function of Dirichlet Problem \ref{pr.Dir}. 

In the following corollary $\{ B_j\}_{j=0}^{2m-1}$  is a Dirichlet system of order 
$(2m-1)$, considered as an extension of the Dirichlet system $\{ B_j\}_{j=0}^{m-1}$. 
As before,  $\{ \tilde B_j^{L}\}_{j=0}^{2m-1}$ is the dual Dirichlet system 
$\{ B_j\}_{j=0}^{2m-1}$ with respect to formula \eqref{eq.Green.M.B} for $A=L$. 

\begin{corollary} \label{c.Dirichlet.Poisson}
Let $L$ be a strongly elliptic differential operator of order $2m$, $m\geq 1$, 
%Let $L = A^*A$ where $A$ is a differential  operator with injective symbol, $m\geq 1$, 
with smooth coefficients  in a neighbourhood $X$ of $\overline D$,  
$\partial D\in C^s$, $s\geq m$. If 
\begin{equation} \label{eq.Dir.Hadamard}
[H^{m}_0 (D)]^k
\cap S_{L} (D)= [H^{m}_0 (D)]^k
\cap S_{L^*} (D) =0 
\end{equation} 
then  Problem \ref{pr.Dir} has one and only one solution. Moreover, in this case the Green function of 
the Dirichlet problem, i.e. the Schwartz 
kernel ${\mathcal G}_D (x,y)$ of the operator 
${\mathcal G}_D$,  is a bilateral fundamental solution to $L$ on $D$ and, 
in particular, the solution $u$ to Problem \ref{pr.Dir} is given by
$$
u(x) = \int_D {\mathcal G}_D (x,y) g(y) dy \, + \, {\mathcal P}_D 
\Big(\oplus_{j=0}^{m-1} u_j \Big)(x),
\, x \in D,
$$
where 
$$
 {\mathcal P}_D \Big(\oplus_{j=0}^{m-1} u_j \Big) (x)=
\int_{\partial D} 
\sum_{j=0}^{m-1} (\tilde B^{L}_{2m-j-1}  {\mathcal G}^{*}_D (x,y))^* u_j (y) d\sigma(y) 
$$
is the Poisson type integral related to Problem \ref{pr.Dir} with the corresponding 
linear bounded operator  
$$
 {\mathcal P}_D :[\oplus_{j=0}^{m-1}  H^{s-j-1/2} (\partial D)]^k \to 
[H^{s} (D) ]^k \cap S_{L} (D).
$$ 
\end{corollary}

\begin{proof} See, for instance, \cite{Mikh76}, \cite[Ch. 5]{Roit96} or elsewhere.
\end{proof}

For generalized Laplacians the Hodge decomposition is more simple.

\begin{corollary} \label{c.Dirichlet.Hodge}
Let $L = A^*A$ where $A$ be a differential  operator with injective symbol 
with smooth coefficients  in a neighbourhood $X$ of $\overline D$, $m\geq 1$, 
$\partial D\in C^s$, $s\geq m$ and 
$B=\{B_0,B_1,\dots, B_{m-1}\}$ be a Dirichlet system of order $(m-1)$ 
on $\partial D$. Then $L^*=L$, 
$$[H^{s} (D) \cap H^{m}_0 (D)]^k\cap S_L (D) = 
%[H^{s} (D) \cap H^{m}_0 (D)]^k\cap S_{L^*} (D)= 
[H^{s} (D) \cap H^{m}_0 (D)]^k\cap S_A (D)
$$
and $\Pi^{(1)}_D = \Pi^{(2)}_D = \Pi_D$ where 
$\Pi_D$ be the $L^2 (D)$-orthogonal projection 
on the finite-dimensional space $[H^{s} (D) \cap H^{m}_0 (D)]^k\cap S_A (D)$. 
If the operator $A$ satisfies (US) on $X$ 
then 
%\begin{equation*}
%\label{eq.Dir.Hadamard.Lapl}
$[H^{s} (D) \cap H^{m}_0 (D)]^k\cap S_A (D) = \{0\}$ 
%\end{equation*}
and Problem \ref{pr.Dir} has one and only one solution. 
\end{corollary}

\begin{proof} See, for instance, \cite{ShTaDual} or elsewhere. 
\end{proof}

\begin{remark} \label{r.m=1}
Note that for the second order Laplacians $L=A^*A$ (i.e. for the case 
$m=1$, the results of 
Lemmata \ref{l.Dir.right}, \ref{eq.dual.Dir}, Corollaries \ref{c.Dirichlet.Poisson}, 
\ref{c.Dirichlet.Hodge}  and Theorem \ref{t.Dirichlet.M} can be 
extended domains with Lipschitz boundaries using the classical method of non-negative 
Hermitian forms, see \cite{Mikh76}  or \cite{ShTaDU}, \cite{ShTSibAdv} 
for the case of even more general problem with domains having 
non-smooth boundaries and with non-smooth boundary operators in (weighted) Sobolev spaces.
\end{remark}

\section{Exterior extension problems for strongly elliptic systems
%generalized Laplacians
}
\label{S.2}

We begin the section with the discussion of the ill-posed 
Cauchy problem for a $(k\times k)$-{\it elliptic} system $L$ 
%generalized elliptic  Lapalcian $A^*A$ 
of order $2m$ with $m\geq 1$ on $X $ 
in a very particular situation. Namely, let $\Omega_0$ and $\Omega_1$ be two 
bounded domains in $X$ such that $\overline \Omega_0 \subset \Omega_1$.

\begin{problem}\label{pr.Cauchy.M}
Let $s \in {\mathbb Z}_+$ and $\partial \Omega_0\in C^s$ if $s\geq m$ or $\partial 
\Omega_0\in C^\infty$ if $s<m$. Let also $B=\{B_0,B_1, \dots B_{2m-1}\}$  be 
a Dirichlet system 
of order $(2m-1)$ on $\partial \Omega_0$. Given functions 
$u_j \in [H^{s-j-1/2} (\partial \Omega_0)]^k$, $0\leq j \leq 2m-1$, find, if possible,  
a vector function $u \in [H^{s} (\Omega _1 \setminus \overline \Omega_0)]^k$ such that 
\begin{equation*} %\label{eq.Cauchy.M}
\left\{ \begin{array}{lll}
L u =0 & {\rm in } & \Omega _1 \setminus \overline \Omega_0,\\
\oplus _{j=0}^{2m-1} B_j u = \oplus _{j=0}^{2m-1} u_j  & \rm{  on }  & \partial \Omega_0.\\
%A^*A u =0 & {\rm in } & \Omega _1 \setminus \overline \Omega_0,\\
%\oplus _{j=0}^{m-1} B_j u = \oplus _{j=0}^{m-1} u_j  & \rm{  on }  & \partial \Omega_0,\\
%\oplus _{j=0}^{m-1} \tilde B_{m-1-j} A u = \oplus _{j=0}^{m-1} u_{j+m}  & \rm{  on }  
%& \partial \Omega_0,\\
\end{array}
\right.
\end{equation*}
%where $\tilde B =\{\tilde B_0,\tilde B_1, \dots \tilde 
%B_{m-1}\}$ is the dual Dirichlet system for $B$.
\end{problem}

As the Cauchy problem is generally ill-posed, the description
of its solvability conditions is rather complicated. 
 It appears that the regularization methods 
(see, for instance, \cite{TikhArsX}) are most effective for  
studying the problem.  However, there are many different ways to 
realize the regularization, see, for instance, 
\cite{Lv1}, \cite{MH74}, \cite{KMF91} for the Cauchy problem related 
to the second order elliptic equations. We follow idea of the book \cite{Tark36}, that
gives a rather full description of solvability conditions for the homogeneous 
elliptic equations. The even order of the system $L$ is unessential 
for Problem \ref{pr.Cauchy.M} but is essential to other extension problems 
consired in this section. 
%, combined with the recent results  
%\cite{FeSh14} for elliptic complexes. 

As we mentioned above, if we assume that both operators $L$ and $L^*$
%the operator $A^*A$  
possess the Unique continuation property (US)
then $L$ %$A^*A$ 
admits a bilateral (left and right) fundamental solution $\Phi (x,y)$, 
see, for instance, \cite[\S 2.3]{Tark95a}. In particular, the following 
(representation)
Green formula holds: for each $u \in [H^{s} (\Omega _1 \setminus \overline \Omega_0)]^k$ we have 
\begin{equation} \label{eq.Green.M.two}
\chi_{\Omega _1 \setminus \overline \Omega_0} u = {\mathcal T}^{(B)}_{\partial \Omega_1} (Bu) 
- {\mathcal T}^{(B)}_{\partial \Omega_0} (Bu),  
%{\mathcal W}^{(B)}_{\partial \Omega_1} (Bu ) + 
%{\mathcal V}^{(B)}_{\partial \Omega_1} (\tilde B A u ) -  
%{\mathcal W}^{(B)}_{\partial \Omega_0} (Bu ) - 
%{\mathcal V}^{(B)}_{\partial \Omega_0} (\tilde B A u ),%+ T_{M,D}  (\Delta_M  u),
\end{equation}
where $\chi_D$ is the characteristic function of the (bounded) domain $D$ in ${\mathbb R}^n$, 
%$$
%T_{D,M} (g) (x) = \int_D \varphi_M (x,y) g (y) dy,
%$$
and 
\begin{equation*}
%\label{eq.}
{\mathcal T}^{(B)}_{S} (\oplus_{j=0}^{2m-1} u_j) = \int_S 
\Big(\sum_{j=0}^{2m-1}  (\tilde B^L_j (y) \Phi^{*} (x,y) )^* u_j (y)\Big) 
d\sigma (y), 
\end{equation*}
%$$
%{\mathcal V}^{(B)}_{S} (\oplus_{j=m}^{2m-1} u_j) = \int_S 
%\Big(\sum_{j=m}^{2m-1}  (B_j (y) \Phi (x,y) )^* u_j (y)
%%u_0  (y) \tilde B_1 (y)  \varphi_M (x,y)-  u_1 (y) %\tilde B_0 (y) \varphi_M (x,y)
%\Big) 
%d\sigma (y), 
%$$
%$$
%{\mathcal W}^{(B)}_{S} (\oplus_{j=0}^{m-1} u_j) = \int_S 
%\Big(\sum_{j=0}^{m-1}  (\tilde B_j (y) A _y\Phi (x,y) )^* u_j (y)
%%u_0  (y) \tilde B_1 (y)  \varphi_M (x,y)-  u_1 (y) %\tilde B_0 (y) \varphi_M (x,y)
%\Big) 
%d\sigma (y),
%$$
with a hypersurface $S$, $x \not \in S$, and the dual Dirichlet system $\tilde B^L$ for $B$ 
with respect to the (first) Green formula for the operator $L$.

Let us formulate a solvability criterion for Problem \ref{pr.Cauchy.M} under  
reasonable assumptions on $S$. 
%Namely, let us assume that 
%$S$ is a relatively open subset of $\partial D$ with a smooth boundary $\partial S$. 
%Then for each set   $\oplus_{j=0}^{m-1} u_j \in H^{s-j-1/2} (S)$, 
%there are vector-functions 
%$\tilde u_j \in H^{s-j-1/2} (\partial D)$, 
%such that $\tilde u_j = u_j$, $0\leq j \leq m-1$. on $S$.
%Let us fix a domain $D^+$ such that $D\cap D^+=\emptyset$ and the set $\Omega = D\cup S \cup %D^+$ is a piece-wise smooth domain. 
We denote by $({\mathcal T}^{(B)}_{\partial \Omega_0}(\oplus_{j=0}^{2m-1}  
u_j))^+$ 
 the restriction of the potential ${\mathcal T}^{(B)}_{\partial \Omega_0}(\oplus_{j=0}^{2m-1}  u_j)$
 onto $\Omega_0$. 

Similarly, the restriction 
of the potential ${\mathcal T}^{(B)}_{\partial \Omega_0} (\oplus_{j=0}^{2m-1}  u_j)$
% ${\mathcal W}^{(B)}_{\partial \Omega_0} (\oplus_{j=0}^{m-1} u_j)$, 
%$({\mathcal V}^{(B)}_{\partial \Omega_0} (\oplus_{j=m}^{2m-1}  u_j)$ 
onto $\Omega_1\setminus \overline \Omega_0$ wil be denoted by 
$({\mathcal T}^{(B)}_{\partial \Omega_0} (\oplus_{j=0}^{2m-1}  u_j))^-$. 
%$({\mathcal W}^{(B)}_{\partial \Omega_0} (\oplus_{j=0}^{m-1} u_j)^-$ and 
%$({\mathcal V}^{(B)}_{\partial \Omega_0} (\oplus_{j=m}^{2m-1}  u_j)^-$, 
%respectively. 
Obviously, 
$$
L {\mathcal T}^{(B)}_{\partial \Omega_0} (\oplus_{j=0}^{2m-1}  u_j) = 0 \mbox{ in } 
\Omega_1 \setminus \partial \Omega_0,
$$
%$$
%A^*A \, 
%({\mathcal W}^{(B)}_{\partial \Omega_0} (\oplus_{j=0}^{m-1}  u_j)^+ =
%A^*A \, 
%({\mathcal V}^{(B)}_{\partial \Omega_0} (\oplus_{j=m}^{2m-1}  u_j)^+ =
%0 \mbox{ in } 
%\Omega_0
%$$
as a parameter dependent integrals, i.e.
$({\mathcal T}^{(B)}_{\partial \Omega_0} (\oplus_{j=0}^{2m-1}  u_j))^+ \in S_{L}(\Omega_0)$ 
and, similarly, $({\mathcal T}^{(B)}_{\partial \Omega_0}(\oplus_{j=0}^{2m-1}  u_j))^- \in S_{L}
(\Omega_1 \setminus \overline \Omega_0)$.
%$({\mathcal W}^{(B)}_{\partial \Omega_0} (\oplus_{j=0}^{m-1}  u_j) ^+, 
%({\mathcal V}^{(B)}_{\partial \Omega_0} (\oplus_{j=m}^{2m-1}  u_j)^+ \in 
%S_{A^*A}(\Omega_0)$ and 
%$({\mathcal W}^{(B)}_{\partial \Omega_0} (\oplus_{j=0}^{m-1}  u_j) ^- , 
%({\mathcal V}^{(B)}_{\partial \Omega_0} (\oplus_{j=m}^{2m-1}  u_j)^- \in 
%S_{A^*A}(\Omega_1 \setminus \overline \Omega_0)$. 

\begin{theorem} \label{t.Cauchy.M}
 Let $L$ be an elliptic operator such that both 
$L$ and $L^*$ satisfy (US)
on $X$ and $\Omega_0$ be a 
bounded domain in $X$. Let also 
$s \in {\mathbb Z}_+$ and $\partial \Omega_0\in C^{\max{(s,2)}}$ 
if $s\geq m$ or $\partial \Omega_0\in C^
\infty$ if $s<m$. 
%and the matrix $M$ have real analytic entries and satisfy \eqref{eq.M.pos}. 
%If $\partial D \setminus S$ has at least one interior point in the relative 
%topology then 
If $\Omega_1 \setminus  \Omega_0$ has no compact components in $\Omega_1$ then 
Problem \ref{pr.Cauchy.M} is densely solvable and 
%If $S$ is a relatively open subset of $\partial D$ with a smooth boundary $\partial S$ then
%Cauchy problem \ref{pr.Cauchy.M} 
it has no more than one solution. It 
is solvable if and only if there is 
a function ${\mathcal F}\in [H^s (\Omega_1)]^k\cap S_{L}(\Omega_1)$ 
%satisfying $\Delta_M {\mathcal F} = 0$ in  $G$
and such that 
$$
{\mathcal F}= ({\mathcal T}^{(B)}_{\partial \Omega_0}(\oplus_{j=0}^{2m-1}  u_j))^+
%({\mathcal W}^{(B)}_{\partial \Omega_0} 
%(\oplus_{j=0}^{m-1}  u_j))^+ + 
%({\mathcal V}^{(B)}_{\partial \Omega_0} (\oplus_{j=m}^{2m-1}  u_j)^+
\mbox{ in } \Omega_0.
$$
Besides, the solution $u$, if exists, is given by the following formula:
\begin{equation} \label{eq.sol.Cauchy}
u  = ({\mathcal T}^{(B)}_{\partial \Omega_0}(\oplus_{j=0}^{2m-1}  u_j))^-
%{\mathcal W}^{(B)}_{\partial \Omega_0} (\oplus_{j=0}^{m-1}  u_j)   
%+ {\mathcal V}^{(B)}_{\partial \Omega_0} (\oplus_{j=m}^{2m-1} u_j) 
- {\mathcal F} \mbox{ in } \Omega_1 \setminus \overline \Omega_0.
\end{equation}
\end{theorem}

\begin{proof} All the statement, except the dense solvability  
follow from \cite[Theorems 2.8 and 5.2]{ShTaLMS}. 
The dense solvability of the Cauchy problem 
was established in spaces of different types, see, 
for instance,  \cite{MH74}, \cite{Me56}, for the Laplace operator in spaces 
of $C^1$-smooth functions or \cite[Lemma 3.2]{ShTaLMS} for general elliptic 
systems in the so-called Hardy spaces. In order to give some arguments 
about the dense solvability in this particular situation 
we may use Approximation Theorems for solutions to elliptic systems,
 see, for instance, \cite[Ch. 5-8]{Tark36} and 
Theorems on the jump behaviour of the potential 
${\mathcal T}^{(B)}_{\partial \Omega_0}(\oplus_{j=0}^{2m-1}  u_j)$, 
%${\mathcal V}_S$, ${\mathcal W}_S$, 
see \cite[Theorem 3.3.9]{Tark95a} and \cite[Lemma 2.7]{ShTaLMS}. 

Indeed, the  continuity of the potentials 
\begin{equation} \label{eq.potentials.cont.+}
%({\mathcal W}^{(B)}_{\partial \Omega_0})^+
({\mathcal T}^{(B)}_{\partial \Omega_0}(\oplus_{j=0}^{2m-1}  u_j))^+
: \oplus_{j=0}^{2m-1} [H^{s-j-1/2} (\partial \Omega _0 )]^k 
\to [H^s (\Omega_0)]^k, 
\end{equation} 
\begin{equation} \label{eq.potentials.cont.-}
({\mathcal T}^{(B)}_{\partial \Omega_0}(\oplus_{j=0}^{2m-1}  u_j))^-
%({\mathcal W}^{(B)}_{\partial \Omega_0})^-
:\oplus_{j=0}^{m-1} [H^{s-j-1/2} (\partial \Omega _0 )]^k 
\to [H^s (\Omega_1 \setminus \overline \Omega_0)]^k, 
\end{equation}
%\begin{equation} \label{eq.potentials.cont.V+}
%({\mathcal V}^{(B)}_{\partial \Omega_0})^+:
%\oplus_{j=0}^{m-1} [H^{s-2m+j+1/2} (\partial \Omega _0 )]^k 
%\to [H^s (\Omega_0)]^k, 
%\end{equation} 
%\begin{equation} \label{eq.potentials.cont.V-}
%({\mathcal V}^{(B)}_{\partial \Omega_0})^-:
%\oplus_{j=0}^{m-1} [H^{s-2m+j+1/2} (\partial \Omega _0 )]^k 
%\to [H^s (\Omega_1 \setminus \overline \Omega_0)]^k, 
%\end{equation}
follows from Theorems on the boundedness of the boundary 
potential operators related the pseudo differential operators satisfying the 
so-called transmission property, see, for instance, 
 \cite[\S 2.3.2.5 ]{RS82} or \cite[\S 2.4]{Tark36}. 

In particular, the continuity of the potentials mean that 
\begin{equation} \label{eq.potentials.plus}
({\mathcal T}^{(B)}_{\partial \Omega_0}(\oplus_{j=0}^{2m-1}  u_j))^+
%({\mathcal W}^{(B)}_{\partial \Omega_0} (\oplus_{j=0}^{m-1}  u_j) ^+, 
%({\mathcal V}^{(B)}_{\partial \Omega_0} (\oplus_{j=m}^{2m-1}  u_j)^+ 
\in S_{L}(\Omega_0) \cap [H^s ( \Omega_0)]^k,
\end{equation}
\begin{equation} \label{eq.potentials.minus}
({\mathcal T}^{(B)}_{\partial \Omega_0}(\oplus_{j=0}^{2m-1}  u_j))^-
%({\mathcal W}^{(B)}_{\partial \Omega_0} (\oplus_{j=0}^{m-1}  u_j) ^- , 
%({\mathcal V}^{(B)}_{\partial \Omega_0} (\oplus_{j=m}^{2m-1}  u_j)^- 
\in S_{L}(\Omega_1 \setminus \overline \Omega_0) \cap 
[H^s (\Omega_1 \setminus \overline \Omega_0)]^k.
\end{equation}

Next, according to Jump Theorems, for $0\leq j\leq 2m-1$ we have
\begin{equation} \label{eq.jump}
B_j ({\mathcal T}^{(B)}_{\partial \Omega_0}(\oplus_{j=0}^{2m-1}  u_j))^-
- B_j ({\mathcal T}^{(B)}_{\partial \Omega_0}(\oplus_{j=0}^{2m-1}  u_j))^+ =u_j \mbox{ on } \partial \Omega_0,
\end{equation}
%\begin{equation} \label{eq.jump.V1}
%(B_j ({\mathcal V}^{(B)}_{\partial \Omega_0})(\oplus_{j=m}^{2m-1}  u_j)))^- 
%- (B_j ({\mathcal V}^{(B)}_{\partial \Omega_0} (\oplus_{j=m}^{2m-1}  u_j)
%))^+ =0 \mbox{ on } \partial \Omega_0,
%\end{equation}
%\begin{equation} \label{eq.jump.V2}
%(\tilde B_{m-1-j} A ({\mathcal V}^{(B)}_{\partial \Omega_0} (\oplus_{j=m}^{2m-1}  u_j)
%))^- - (\tilde B_{j-1-j} A ({
%\mathcal V}^{(B)}_{\partial \Omega_0} (\oplus_{j=m}^{2m-1}  u_j)
%))^+ =u_{m+j} \mbox{ on } \partial \Omega_0, 
%\end{equation}
%\begin{equation} \label{eq.jump.W1}
%(B_j ({\mathcal W}^{(B)}_{\partial \Omega_0} (\oplus_{j=0}^{m-1}  u_j)
%))^- - (B_j ({\mathcal W}^{(B)}_{\partial \Omega_0}
%(\oplus_{j=0}^{m-1}  u_j) ))^+ =u_j \mbox{ on } \partial \Omega_0, 
%\end{equation}
%\begin{equation} \label{eq.jump.W2}
%(\tilde B_{m-1-j} A ({\mathcal W}^{(B)}_{\partial \Omega_0}
%(\oplus_{j=0}^{m-1}  u_j) ))^- - (\tilde B_{j-1-j} A ({
%\mathcal W}^{(B)}_{\partial \Omega_0} (\oplus_{j=0}^{m-1}  u_j)))^+ =0  \mbox{ on } \partial %\Omega_0, 
%\end{equation}
where the type of boundary values depends on the order of the differential operators.
Of course, if $s\geq 2m$ then all the boundary values under the 
consideration  can be treated as the standard traces on $\partial \Omega$. Otherwise, 
if the orders of operators $B_j$ 
are greater than $s$ 
then  we should pass to the so-called weak boundary values, see, for instance, \cite{Str84} 
for harmonic functions or \cite[\S\S 9.3, 9.4]{Tark97} or \cite[Definition 2.2]{ShTaLMS} 
for general elliptic systems satisfying the Unique Continuation Property (US) 
and then the boundary should be sufficiently regular (for example, $C^\infty$-smooth for $0
\leq s<m$, but it always can be a high finite smoothness). In any case, 
under the assumptions above the mappings
\begin{equation} \label{eq.traces.int}
 B_j : [H^s (\Omega_0)]^k \cap S_{L}(\Omega_0) \to 
 [H^{s-j-1/2} (\partial \Omega _0 )]^k, \, 
\end{equation}
\begin{equation} \label{eq.traces.ext}
B_j : [H^s (\Omega_1 \setminus \overline \Omega_0)]^k \cap 
S_{L}(\Omega_1 \setminus \overline \Omega_0) \to 
[H^{s-j-1/2} (\partial \Omega _0 )]^k ,
\end{equation}
%\begin{equation} \label{eq.traces.high.int}
%\tilde B_{m-j-1} A : [H^s (\Omega_0)]^k \cap S_{A^*A}(\Omega_0)  \to 
%[H^{s-2m+j+1/2} (\partial \Omega _0 )]^k ,
%\end{equation}
%\begin{equation} \label{eq.traces.high.ext}
%\tilde B_{m-j-1} A : [H^s (\Omega_1 \setminus \overline \Omega_0)]^k \cap 
%S_{A^*A}(\Omega_1 \setminus \overline \Omega_0) \to 
% [H^{s-2m+j+1/2} (\partial \Omega _0 )]^k 
%\end{equation}
are continuous. 

Next, for a domain $D$ in $X$ we denote by $S_{L} (\overline D)$ the union $\cup
_{U\supset \overline D}S_{L} (U)$ of solutions on all open sets in $X$ 
containing  $\overline D$. 

According to \cite[Theorem 8.2.2]{Tark36} for $s< 2m$ and 
\cite[Theorems 8.1.2 and 8.1.3]{Tark36} for $s\geq 2m$, 
the space $S_{L} (\overline \Omega_0)$ is everywhere dense in 
$S_{L} (\Omega_0) \cap [H^s (\Omega_0)]^k$ and, similarly, 
the space $S_{L} (\overline \Omega_1 \setminus \Omega_0)$ is everywhere dense in the space 
$S_{L} (\Omega_1\setminus \overline \Omega_0) \cap 
[H^s (\Omega_1\setminus \overline \Omega_0)]^k$. In particular, 
the potential $({\mathcal T}^{(B)}_{\partial \Omega_0}(\oplus_{j=m}^{2m-1}  u_j))^-$
%$({\mathcal W}^{(B)}_{\partial \Omega_0} (\oplus_{j=0}^{m-1}  u_j) ^+ + 
%({\mathcal V}^{(B)}_{\partial \Omega_0} (\oplus_{j=m}^{2m-1}  u_j)^+$ 
can be approximated 
in the space $[H^s (\Omega_1 \setminus \overline  \Omega_0)]^k$ by elements from 
$S_{L} (\overline \Omega_1 \setminus \Omega_0)$. 

The Runge type Theorems for solutions to elliptic systems, 
see, for instance, \cite{Me56}	for harmonic functions, 
\cite[Theorems 5.1.11, 5.1.13]{Tark36}, yields that the space 
$S_{L} (\overline \Omega_1)$ is everywhere dense in 
the space $S_{L} (\overline \Omega_0)$, since 
$\Omega_1 \setminus  \Omega_0$ has no compact components in $\Omega_1$. 
Hence, relations \eqref{eq.potentials.plus}, \eqref{eq.potentials.minus} 
imply that the potential $({\mathcal T}^{(B)}_{\partial \Omega_0}(\oplus_{j=m}^{2m-1}  u_j))^+$
%$({\mathcal W}^{(B)}_{\partial \Omega_0} (\oplus_{j=0}^{m-1}  u_j) ^- + 
%({\mathcal V}^{(B)}_{\partial \Omega_0} (\oplus_{j=m}^{2m-1}  u_j)^-$ 
can be approximated 
in the space $[H^s ( \Omega_0)]^k$ by elements from $S_{L} (\overline \Omega_1)$. 
%and, similarly, the potentials 
%$({\mathcal W}^{(B)}_{\partial \Omega_0} (\oplus_{j=0}^{m-1}  u_j) ^+, 
%({\mathcal V}^{(B)}_{\partial \Omega_0} (\oplus_{j=m}^{2m-1}  u_j)^+$  
%can be approximated 
%in the space $S_{A^*A}(\Omega_1 \setminus \overline\Omega_0) \cap [H^s ( 
%\Omega_1 \setminus \overline \Omega_0)]^k$ 
%by elements from $S_{A^*A}(\overline \Omega_1 \setminus \Omega_0)$.

Finally, the jump formulas \eqref{eq.jump} %\eqref{eq.jump.V1}-\eqref{eq.jump.W2} 
combined with the continuity relations \eqref{eq.traces.int}--\eqref{eq.traces.ext} 
yield the possibility to approximate the Cauchy data $\oplus _{j=0}^{2m-1} u_j $ 
in the space $ \oplus _{j=0}^{2m-1} [H^{s-j-1/2} ( \partial \Omega_0)]^k $  
by elements of the type $\oplus _{j=0}^{2m-1} B_j u $ with $u $ from the space 
$S_{L}(\overline \Omega_1 \setminus \Omega_0)$, which was to be proved. 
\end{proof}

Thus, Theorem \ref{t.Cauchy.M} reduces Problem \ref{pr.Cauchy.M}
to the following problem related to "analytic continuation"{} from an 
open subset of $X$ to a larger one.

\begin{problem} \label{pr.ext.1} Given $V \in [H^s (\Omega_0)]^k \cap 
S_{L} ( \Omega_0)$ find, if possible, $U 
\in [H^s (\Omega_1)]^k \cap S_{L} ( \Omega_1)$ such that $U=V$ in $\Omega_0$.
\end{problem}

\begin{corollary} \label{c.ext.1}
Let $L$ be an elliptic operator such that both $L$ and $L^*$ satisfy (US) on $X$ and $\Omega_0$ be a 
bounded domain. Let also 
$s \in {\mathbb Z}_+$ and $\partial \Omega_0\in C^{\max{(s,2)}}$ 
if $s\geq m$ or $\partial \Omega_0\in C^
\infty$ if $s<m$. 
%and the matrix $M$ have real analytic entries and satisfy \eqref{eq.M.pos}. 
%If $\partial D \setminus S$ has at least one interior point in the relative 
%topology then 
If $\Omega_1 \setminus  \Omega_0$ has no compact components in $\Omega_1$ then 
Problem \ref{pr.ext.1} is densely solvable and 
%If $S$ is a relatively open subset of $\partial D$ with a smooth boundary $\partial S$ then
%Cauchy problem \ref{pr.Cauchy.M} 
it has no more than one solution. It 
is solvable if and only if  Problem \ref{pr.Cauchy.M} 
is solvable for the data $\oplus_{j=0}^{2m-1} u_j = \oplus_{j=0}^{2m-1} B_j V$ %and 
%$\oplus_{j=0}^{m-1} u_{m+j}= \oplus_{j=1}^m \tilde B_{m-1-j} A V$ 
on $\partial \Omega_0$. 
%%{\mathcal P} (\oplus_{j=1}^m B_j u)$  
\end{corollary}

\begin{proof} The uniqueness of the problem is provided by 
the Unique Continuation Property (US) on $X$. The dense solvability 
follows from the same arguments as in the proof of Theorem \ref{t.Cauchy.M}: 
first,  \cite[Theorem 8.2.2]{Tark36} for $s< 2m$ and 
\cite[Theorems 8.1.2 and 8.1.3]{Tark36} for $s\geq 2m$, implies that 
the space $S_{L} (\overline \Omega_0)$ is everywhere dense in 
$S_{L} (\Omega_0) \cap [H^s (\Omega_0)]^k$ and then 
the Runge type Theorems for solutions to elliptic systems, 
see, for instance, \cite{Me56}	for harmonic functions, 
\cite[Theorems 5.1.11, 5.1.13]{Tark36}, provides that the space 
$S_{L} (\overline \Omega_1)$ is everywhere dense in 
the space $S_{L} (\overline \Omega_0)$, since 
$\Omega_1 \setminus  \Omega_0$ has no compact components in $\Omega_1$. 

If $U$ is a solution to Problem \ref{pr.ext.1} then $U \in [H^{s} (\Omega_1)]^k 
\cap S_{L} (\Omega_1)$, and 
the Cauchy data $\oplus_{j=0}^{2m-1}  B_j U = \oplus_{j=0}^{2m-1}  B_j V $,  
%$\oplus_{j=0}^{m-1} \tilde B_{m-1-j} A U = \oplus_{j=0}^{m-1} \tilde B_{m-1-j} A V$ 
are well-defined on $\partial \Omega_0$.  Obviously the restriction $u = U|\Omega_1 $ is the 
unique solution to Problem \ref{pr.Cauchy.M}.

If $u \in [H^{s} (\Omega_1 \setminus \overline \Omega_0)]^k$ is 
the solution to  Problem \ref{pr.Cauchy.M} with the Cauchy data 
$\oplus_{j=0}^{2m-1}  u_j = \oplus_{j=1}^m  B_j V $,  
%$\oplus_{j=0}^{m-1} u_{j+m} = \oplus_{j=0}^{m-1} \tilde B_{m-1-j} A V$ 
 on $\partial \Omega_0$ then the function 
$$
U = \left\{
\begin{array}{lll} V & \rm{in} & \Omega_0, \\
u & \rm{in} & \Omega_1 \setminus \overline \Omega_0
\end{array}
\right.
$$
belongs to $ ([H^{s} (\Omega_0)]^k 
\cap S_{L} (\Omega_0)) \cap 
([H^{s} (\Omega_1 \setminus \overline \Omega_0)]^k 
\cap S_{L} (\Omega_1\setminus \overline  \Omega_0)) $ 
and satisfies 
\begin{equation}\label{eq.erasing.1}
\oplus_{j=0}^{2m-1} B_j U^- =
\oplus_{j=0}^{2m-1} B_j u =  \oplus_{j=0}^{m-1} B_jV  = 
\oplus_{j=0}^{m-1} B_j U^+ \mbox{ on } \partial \Omega_0, 
\end{equation}
%\begin{equation}\label{eq.erasing.2}
%\oplus_{j=0}^{m-1} \tilde B_{m-1-j} A   U ^-=
%\oplus_{j=0}^{m-1} \tilde B_{m-1-j} A  u =  \oplus_{j=0}^{m-1} \tilde B_{m-1-j} A  
%V = \oplus_{j=0}^{m-1} \tilde B_{m-1-j} A   U ^+ 
%\end{equation}
 on  $\partial \Omega_0$, too. 
Then, by the  Green formula for solutions to elliptic systems \eqref{eq.Green.M.two}, 
see \cite[Theorem 2.4]{ShTaLMS}, 
\begin{equation}\label{eq.Green.-}
\chi_{\Omega_1\setminus 
\overline \Omega_0} u = {\mathcal T}^{(B)}_{\partial (\Omega_1\setminus 
\overline \Omega_0)} (\oplus_{j=0}^{2m-1}  B_j u) 
%{\mathcal W}^{(B)}_{\partial (\Omega_1\setminus 
%\overline \Omega_0)} (\oplus_{j=0}^{m-1}  B_j u) + 
%{\mathcal V}^{(B)}_{\partial (\Omega_1\setminus 
%\overline \Omega_0)} (\oplus_{j=0}^{m-1} \tilde B_{m-j-1} Au) 
\mbox{ in } \Omega_1 \setminus \partial \Omega_0, 
\end{equation}
\begin{equation}\label{eq.Green.+}
\chi_{ \Omega_0}
V = {\mathcal T}^{(B)}_{\partial \Omega_0} (\oplus_{j=0}^{2m-1}  B_j V) 
%{\mathcal W}^{(B)}_{\partial \Omega_0} (\oplus_{j=0}^{m-1}  B_jV)  + 
%{\mathcal V}^{(B)}_{\partial \Omega_0} (\oplus_{j=0}^{m-1} \tilde B_{m-j-1}AV)
%\mbox{ in } 
%\Omega_1 \setminus \partial \Omega_0, 
\end{equation}
and then, adding \eqref{eq.Green.+} to \eqref{eq.Green.-} and  
taking in account relations \eqref{eq.erasing.1} %\eqref{eq.erasing.2} 
and the orientation of the boundaries of the domains 
$\Omega_0$ and  $\Omega_1 \setminus  \Omega_0$ we see that 
\begin{equation}\label{eq.Green.add}
U = {\mathcal T}^{(B)}_{\partial \Omega_1} (\oplus_{j=0}^{2m-1}  B_j u)
%{\mathcal W}^{(B)}_{\partial \Omega_1} (\oplus_{j=0}^{m-1}  B_ju) + 
%{\mathcal V}^{(B)}_{\partial \Omega_1} (\oplus_{j=0}^{m-1} \tilde B_{m-j-1} Au)  
\mbox{ in } \Omega_1 \setminus \partial \Omega_0. 
\end{equation}
The right hand side of \eqref{eq.Green.add} belongs to $S_{L} (\Omega_1)$ 
as a parameter depending integral and hence $U \in 
[H^s (\Omega_1)]^k \cap S_{L} (\Omega_1)$ is the unique solution to 
Problem \ref{pr.ext.1}. 
\end{proof}

Let us consider one more "extension problem"{} that can be called an "inner"{} 
Dirichlet problem for a strongly elliptic operator $L$.

\begin{problem} \label{pr.ext.2} 
Given Dirichlet data $\oplus_{j=0}^{m-1}
v_j \in \oplus_{j=0}^{m-1} [H^{s-j-1/2} (\partial \Omega_0)]^k $ 
find, if possible, $v 
\in [H^s (\Omega_1)]^k \cap S_{L} ( \Omega_1)$ such that $
\oplus_{j=0}^{m-1} B_j v = 
\oplus_{j=0}^{m-1} v_j$ on $\partial \Omega_0$.
\end{problem}

Though it depends on the Dirichlet data, this problem is closely  
related to Problems \ref{pr.Cauchy.M} and  \ref{pr.ext.2}. 

\begin{corollary} \label{c.ext.2}
Let $L$ be a strongy elliptic operator such that both $L$ and $L^*$ satisfy 
(US)  on $X$ and $\Omega_0$ be a 
bounded domain. Let also 
$s \in {\mathbb Z}_+$ and $\partial \Omega_0\in C^{\max{(s,2)}}$ 
if $s\geq m$ or $\partial \Omega_0\in C^
\infty$ if $s<m$. 
%and the matrix $M$ have real analytic entries and satisfy \eqref{eq.M.pos}. 
%If $\partial D \setminus S$ has at least one interior point in the relative 
%topology then 
If \eqref{eq.Dir.Hadamard} holds and 
$\Omega_1 \setminus  \Omega_0$ has no compact components in $\Omega_1$ then 
Problem \ref{pr.ext.2} is densely solvable and 
%If $S$ is a relatively open subset of $\partial D$ with a smooth boundary $\partial S$ then
%Cauchy problem \eqref{pr.Cauchy.M} 
it has no more than one solution. Moreover, the following 
statements are equivalent: 
\begin{itemize}
\item
Problem \ref{pr.ext.2}  is solvable with the %Dirichlet 
data 
$(\oplus_{j=0}^{m-1} v_j) \in \oplus_{j=0}^{m-1} [H^{s-j-1/2} (\partial \Omega_0)]^k $;
\item
Problem \ref{pr.ext.1}  is solvable with $V= {\mathcal P}_{\Omega_0} 
(\oplus_{j=0}^{m-1} v_j)$;
\item 
 Problem \ref{pr.Cauchy.M} 
is solvable for the Cauchy data $\oplus_{j=0}^{2m-1} u_j $ where 
$\oplus _{j=0}^{m-1}u_j =\oplus_{j=0}^{m-1} v_j$ and   
$\oplus_{j=m}^{2m-1} u_{j}= \oplus_{j=m}^{2m-1}  B_j {\mathcal P}_{\Omega_0} 
(\oplus_{j=0}^{m-1} v_j)$ on $\partial \Omega_0$.  
\end{itemize}
\end{corollary}

\begin{proof} Indeed, according to Corollary \ref{c.Dirichlet.Poisson}, 
the Poisson integral ${\mathcal P}_{\Omega_0}$ induces 
an isomorphism of Banach spaces
$$
{\mathcal P}_{\Omega_0}: \oplus_{j=0}^{m-1} [H^{s-j-1/2} (\partial \Omega_0)]^k
\to [H^{s} (\Omega_0)]^k \cap S_{L} ( \Omega_0).
$$
For this reason any solution $U$ to Problem \ref{pr.ext.2} 
coincides with  ${\mathcal P}_{\Omega_0} 
(\oplus_{j=0}^{m-1} v_j)$ in $\Omega_0$ and this proves the equivalence 
between the first two statements of the corollary. In particular, 
the Uniqueness Theorem for Problem \ref{pr.ext.2} and its  
dense solvability immediately follow from Corollary \ref{c.ext.1}.

Moreover, as the operator 
$$
\oplus_{j=m}^{2m-1} B_{j}
{\mathcal P}_{\Omega_0}: \oplus_{j=0}^{m-1} [H^{s-j-1/2} (\partial \Omega_0)]^k 
\to \oplus_{j=m}^{2m-1} [H^{s-j-1/2} (\partial \Omega_0)]^k
$$
is well-defined (of course, as we mentioned above, it should be understood 
in the sense of weak boundary values for $s<2m$), because of the 
continuity of mapping \eqref{eq.traces.int}, \eqref{eq.traces.ext}; 
it represents the so-called Dirichlet-to-Neumann operator in a sense. 
Combined with Theorem \ref{t.Cauchy.M}, 
this proves the equivalence of the second and the third statements 
of the corollary.  
\end{proof}

Thus, Corollary \ref{c.ext.2} hints us that Problem \ref{pr.ext.1} is 
of key importance for Problems \ref{pr.Cauchy.M} and \ref{pr.ext.2}.
Its solvability conditions and formulas for approximate and exact solutions
may be indicated in terms of the so-called bases with the double orthogonality 
property, see \cite{ShTaLMS} and \cite[Ch. 12]{Tark36}. Hence the Cauchy 
Problem  \ref{pr.Cauchy.M} and interior Dirichlet Problem \ref{pr.ext.2} can 
be investigated with the use of these bases; of course, one may 
elaborate some iteration algorithms, see, for instance, \cite{KMF91} or envoke 
integral representation method, see \cite{A}, \cite{Lv1} or elsewhere.  

We finish the section with another important issue, namely, the conditional well-posedness (see, for instance, 
\cite{TikhArsX}) of the problems considered above.  With this purpose, 
given positive constant $\gamma$, denote by $[H^s_{L,\gamma} (\Omega_1)]^k$
the set of vectors $U 
\in [H^s (\Omega_1)]^k \cap S_{L} ( \Omega_1)$ such that 
$$
\| U\|_{[H^s (\Omega_1)]^k} \leq \gamma.
$$
We consider it as a metric space with the metric induced by 
the norm $\| \cdot \|_{[H^s (\Omega_1)]^k}$.

\begin{corollary} \label{c.well.ext.1} 
Let $L$ be an elliptic operator such that both $L$ and $L^*$ satisfy (US)
on $X$ and $\Omega_0$ be a 
bounded domain. Let also $s \in {\mathbb Z}_+$ and $\partial \Omega_0\in 
C^{\max{(s,2)}}$ 
if $s\geq m$ or $\partial \Omega_0\in C^
\infty$ if $s<m$. If $\Omega_1 \setminus  \Omega_0$ has no compact components in $\Omega_1$ 
then  Problems \ref{pr.Cauchy.M},  \ref{pr.ext.1} 
are conditionally well-posed in the following senses:
\begin{enumerate}
\item
if a sequence $\{U_i\} \subset [H^s_{L,\gamma} (\Omega_1)]^k$
converges to zero in  $[H^s ( \Omega_0)]^k$ then 
$\{U_i\} $ converges to zero in the local space $[H_{\rm loc}^{s} ( \Omega_1 )]^k$ 
and weakly converges to zero in $[H^{s} ( \Omega_1 )]^k$;
\item
if for a sequence $\{u_i\} \subset [H^s_{L,\gamma}(\Omega_1\setminus \overline\Omega_0)]^k$
the sequences $\{B_j u_i\}$ converge to zero in  $[H^{s-j-1/2} (\partial \Omega_0)]^k$ 
%and  sequences $\{\tilde B_{m-j-1} A u_i\}$ converge to zero in  $[H^{s-2m+j+1/2}
% (\partial \Omega_0)]^k$ 
for each $j$, $0\leq j\leq 2m-1$, then 
$\{u_i\} $ converges to zero in the space $[H_{\rm loc}^{s} ( \Omega_1 
\setminus \Omega_0)]^k$.
\end{enumerate}
If, moreover, $L$ is a strongly elliptic operator satisfying \eqref{eq.Dir.Hadamard} 
then Problem \ref{pr.ext.2}  is conditionally well-posed in the following sense: 
\begin{enumerate}
\item[(3)]
if for a sequence $\{U_i\} \subset [H^s_{L,\gamma}(\Omega_1)]^k$ the sequences $\{B_j U_i\}$ 
converge to zero in  $[H^{s-j-1/2} (\partial \Omega_0)]^k$ for all $j$, $0\leq j \leq m-1$, 
then $\{U_i\} $ converges to zero in the space $[H_{\rm loc}^{s} ( \Omega_1 )]^k$.
\end{enumerate}
%and weakly in $[H^s ( \Omega_1)]^k$.
\end{corollary}

\begin{proof} We begin with  Problem \ref{pr.ext.1}.  
Actually, the topology of the space 
$[H_{\rm loc}^{s} ( \Omega_1)]^k$ is "metrizable". Namely, 
one may fix  an exhaustion $\{ D_\nu\}$ of the domain $ \Omega_1 $ 
by relatively compact domains $D_\nu \subset \Omega_1$, $\cup _\nu \overline D_\nu = \Omega_1$ 
and define a metric as follows:
$$
\rho_{\Omega_1} (U,V) = \sum_{\nu=1}^\infty \frac{1}{2^\nu} 
\frac{\|U-V\|_{[H^{s} ( D_\nu)]^k}}{1+\|U-V\|_{[H^{s} ( D_\nu)]^k}}.
$$
If a sequence $\{U_i\}$ belongs to $ H^s_{L,\gamma} (\Omega_1)$ then it is 
bounded in this space. Then by the weak compactness principle 
one may extract a subsequence $\{U'_{i}\}$ weakly convergent 
to a vector $U_0 \in [H^{s} ( \Omega_1)]^k$. Now  
Stieltjes-Vitali Theorem for solutions to elliptic systems implies that the 
subsequence $\{U'_{i}\} $ converges to $U_0$, 
in fact, in the space $C^\infty _{\rm loc} (\Omega_1)$; in particular, this means 
$U_0 \in [H^{s} ( \Omega_1)]^k \cap S_{L} (\Omega_1)$.

On the other hand, according to the hypothesis of the theorem it converges 
to zero in the space $[H^{s} ( \Omega_0)]^k$. Hence $U_0 =0$ in 
$\Omega_0$ and, by the Unique Continuation Property, $U_0 \equiv 0$ in $\Omega_1$.

Now, if the sequence $\{U_i\}$ does not converge to zero in 
$[H_{\rm loc}^{s} ( \Omega_1 )]^k$ then there is a positive number $\varepsilon_0$
and a subsequence $\{\tilde U_{i}\} $ such that 
\begin{equation} \label{eq.contr}
\rho_{\Omega_1}(\tilde U_{i},0) \geq \varepsilon_0. 
\end{equation}
Applying to the sequence $\{\tilde U_{i}\} $ the arguments as above 
we may extract a subsequence $\{\tilde U_{i}\}$ weakly convergent to zero in 
$[H^{s} ( \Omega_1)]^k $ and convergent to zero in the space $C^\infty _{\rm loc} (\Omega)$, 
obtaining a contradiction with \eqref{eq.contr}.

Similarly, if the sequence $\{U_i\}$ does not converge to zero weakly in 
$[H^{s} ( \Omega_1 )]^k$ then it has a subsequence weakly converging to a non-zero 
element $U_0 \in [H^{s} ( \Omega_1 )]^k \cap \cap S_{L} (\Omega_1)$. 
Repeating the arguments above we see that $U_0\equiv 0$ in $\Omega_1$, i.e. we arrive at 
the contradiction.

We continue with  Problem \ref{pr.Cauchy.M}. 

If a sequence $\{u_i\}$ belongs to $ [H^s_{L,\gamma} 
(\Omega_1 \setminus \overline\Omega_0)]^k$ 
then it is bounded in this space. Each element $u_i$ of this sequence 
is the unique solution to the Cauchy problem \ref{pr.Cauchy.M} with the data 
$\oplus_{j=0}^{2m-1} u_{j,i} $ with $u_{j,i} = B_j u_i$ 
% and $u_{m+j,i} =\tilde B_{m-j-1} A u_i$ 
for $0\leq j \leq 2m-1$. 
Then Theorem \ref{t.Cauchy.M} implies that for the  potential 
$({\mathcal T}^{(B)}_{\partial \Omega_0} (\oplus_{j=0}^{2m-1}  u_{j,i}))^+$ 
%$({\mathcal W}^{(B)}_{\partial \Omega_0} 
%(\oplus_{j=0}^{m-1}  u_{j,i}))^+ + 
%({\mathcal V}^{(B)}_{\partial \Omega_0} (\oplus_{j=m}^{2m-1}  u_{j,i})^+$ 
there is the unique 
extension ${\mathcal F}_i \in H^s (\Omega_1) \cap S_{L} (\Omega_1)$ in the domain 
$\Omega_1$. Moreover, formula \eqref{eq.sol.Cauchy} means 
that 
\begin{equation} \label{eq.Fnu}
{\mathcal F}_i = ({\mathcal T}^{(B)}_{\partial \Omega_0} (\oplus_{j=0}^{2m-1}  u_{j,i}))^-
%{\mathcal W}^{(B)}_{\partial \Omega_0} (\oplus_{j=0}^{m-1}  u_{j,i}) + 
%{\mathcal V}^{(B)}_{\partial \Omega_0} (\oplus_{j=m}^{2m-1}  u_{j,i}) 
- \chi_{\Omega_1\setminus \overline \Omega_0} u_i. 
\end{equation}
If the sequences $\{B_j u_i\}$ converge to zero in  $[H^{s-j-1/2} (\partial \Omega_0)]^k$ 
%and sequences $\{\tilde B_{m-j-1} A u_i\}$ converge to zero in  $[H^{s-2m+j+1/2} (\partial \Omega_0)]^k$ 
for each $j$, $0\leq j\leq 2m-1$, then the continuity of the potentials, see 
\eqref{eq.potentials.cont.+}, \eqref{eq.potentials.cont.-}, 
%-- \eqref{eq.potentials.cont.V-}, 
implies
\begin{equation} \label{eq.potential+.to.zero}
({\mathcal T}^{(B)}_{\partial \Omega_0} (\oplus_{j=0}^{2m-1}  u_{j,i}))^+
%({\mathcal W}^{(B)}_{\partial \Omega_0}(\oplus_{j=0}^{m-1}  u_{j,i}))^+ + 
%({\mathcal V}^{(B)}_{\partial \Omega_0} (\oplus_{j=m}^{2m-1}  u_{j,i}))^+
\to 0 \mbox{ in } H^s (\Omega_0) ,
\end{equation}
\begin{equation} \label{eq.potential-.to.zero}
({\mathcal T}^{(B)}_{\partial \Omega_0} (\oplus_{j=0}^{2m-1}  u_{j,i}))^-
%({\mathcal W}^{(B)}_{\partial \Omega_0}(\oplus_{j=0}^{m-1}  u_{j,i}))^- + 
%({\mathcal V}^{(B)}_{\partial \Omega_0} (\oplus_{j=m}^{2m-1}  u_{j,i}))^-
\to 0 \mbox{ in } H^s (\Omega_1 \setminus \overline \Omega_0).
\end{equation}
In particular, $\|{\mathcal F}_i \|_{H^s (\Omega_0) }$ converges to zero. 

On the other hand, 
$$
\|{\mathcal F}_i\|^2_{H^s (\Omega_1) } = 
\|{\mathcal F}_i\|^2_{H^s (\Omega_0) } +
\|{\mathcal F}_i\|^2_{ H^s (\Omega_1 \setminus \overline \Omega_0) }\leq 
$$
$$
\|u_i\|^2_{ H^s (\Omega_1 \setminus \overline \Omega_0) } +
2 \|({\mathcal T}^{(B)}_{\partial \Omega_0} (\oplus_{j=0}^{2m-1}  u_{j,i}))^-
%({\mathcal W}^{(B)}_{\partial \Omega_0}(\oplus_{j=0}^{m-1}  u_{j,i}))^- + 
%({\mathcal V}^{(B)}_{\partial \Omega_0} (\oplus_{j=m}^{2m-1}  u_{j,i}))^-
\|^2_{ H^s (\Omega_1 \setminus \overline \Omega_0) } + 
%$$
%$$
2\|({\mathcal T}^{(B)}_{\partial \Omega_0} (\oplus_{j=0}^{2m-1}  u_{j,i}))^+
\|^2_{ H^s (\Omega_0) }
$$
and then, by \eqref{eq.Fnu}, \eqref{eq.potential+.to.zero}, 
\eqref{eq.potential-.to.zero} and the hypothesis of the corollary, 
the sequence $\{ {\mathcal F}_i \}$ belongs to the space
$H^s_{L,\tilde \gamma}(\Omega_1)$ with some $\tilde \gamma \geq \gamma >0$.

At this point, the already proved part (1) of this corollary 
yields that the sequence $\{{\mathcal F}_i\} $ converges to zero in the local space 
$[H_{\rm loc}^{s} ( \Omega_1 )]^k$.  

Hence, by formula \eqref{eq.sol.Cauchy} and \eqref{eq.potential-.to.zero}, we conclude 
that the sequence $\{ u_i\}$ converges to zero in 
the space $[H_{\rm loc}^{s} ( \Omega_1 \setminus \Omega_0)]^k$.

Finally, for Problem \ref{pr.ext.2} we argue as follows. 
Each element $U_i$ is the unique solution to 
Problem \ref{pr.ext.2}  with the Dirichlet data $\{\oplus_{j=0}^{m-1}
u_{j,i} = \oplus_{j=0}^{m-1}B_j U_i\}$. 
Moreover, by Corollary \ref{c.Dirichlet.Poisson} we have 
$U_i = {\mathcal P}_{\Omega_0} (\oplus_{j=0}^{m-1} u_{j,i})$. 
As the Dirichlet problem \ref{pr.Dir} for the operator $L$ 
is well-posed under the hypothesis of this corollary, we see that ${\mathcal P}_{\Omega_0} (\oplus_{j=0}^{m-1} u_{j,i})$ 
converges to zero in $[H^s (\Omega_0)]^k$ if the sequences $\{ B_j U_i\}$
 converge to zero in  $[H^{s-j-1/2} (\partial \Omega_0)]^k$ for each $j$, $0\leq j \leq m-1$.  
Now, the already proved part (1) of this corollary 
yields that the sequence 
$\{U_i\} $ converges to zero in the local space $[H_{\rm loc}^{s} ( \Omega_1 )]^k$.  
\end{proof}

\section{On representation of solutions %of second order strongly elliptic 
%systems 
by the single layer potential }
\label{S.3}

Now we would like to discuss the possibility to 
representation of solutions of second order strongly elliptic 
systems by a single layer potential. For harmonic functions 
of different classes the problem is known since \cite{Ga36} (see also 
 \cite{Solom}, \cite{MH74} and elsewhere). The matter is closely related 
to the theory of multidimensional singular integral equation, see, 
for instance, \cite{Mikhl65}, \cite{Hs91}.

We restrict ourselves with the second order strongly elliptic operator 
with smooth coefficients. Thus, without loss of generality we may assume that 
the operator $L$ is written in the divergence form: 
\begin{equation*} % \label{eq.L.div.form}
L u = - \sum_{i,j=1}^n \partial_j (L_{i,j}  \partial_i u) +
\sum_{i=1}^n L_{i}  \partial _i u + L_0 u
\end{equation*}
with $(k\times k)$-matrices of smooth functions 
$L_{i,j} (x)$, $L_{j} (x)$, $L_{0} (x)$ on $\overline X \subset {\mathbb R}^n$. 
Of course, any generalized Laplacian $A^*A$ is automatically of the divergence form.

Fix a relatively compact domain $\Omega$ in $X$ with the Lipschitz boundary.
Following \cite{Cost88}, we additionally assume that $L$
satisfy following Korn type inequality that is much stronger 
than G\aa{}rding inequality \eqref{eq.Garding} in $\Omega$: 
\begin{equation} \label{eq.Korn}
\|u\|^2_{[H^1 (\Omega)]^k} - c_2 \|u\|^2_{[L^2 (\Omega)]^k}  
 \leq 
\end{equation} 
\begin{equation*}
c_1 \Re{ 
\Big(\sum_{i,j=1}^n (L_{i,j} \partial _i u , \partial_j u )_{[L^2 (\Omega)]^k} +
\big(\sum_{i=1}^n L_{i} \partial _i u + L_0 u,  u \big)_{[L^2 (\Omega)]^k} \Big)} 
\end{equation*}
for all $u \in [H^1 (\Omega)]^k$ 
with some positive constants $c_1$, $c_2$ independent on $u$. 
It is always fulfilled for uniformly strongly elliptic 
scalar operators with real entries. 
For a system $L$ or even for a scalar linear operator with complex-valued 
coefficients \eqref{eq.Korn} can be unaffordable. Indeed, 
for the generalized second order Laplacians $A^*A$ this is equivalent to the following:
\begin{equation} \label{eq.Korn.Laplace}
\|u\|^2_{[H^1 (\Omega)]^k}  \leq c_1 
\|A u \|^2_{[L^2 (\Omega)]^k} + 
c_2 \|u\|^2_{[L^2 (\Omega)]^k} 
\mbox{ for all } u \in [H^1 (\Omega)]^k
\end{equation}
with some positive constants $c_1$, $c_2$ independent on $u$.
Then the presence or absence of \eqref{eq.Korn.Laplace} depends 
on the operator $A$ involved in the factorisation. For instance, 
if $\Delta $ is the usual Laplace operator in ${\mathbb R}^{2}$ 
then $(-\Delta) = \nabla^* \nabla $ and \eqref{eq.Korn.Laplace} holds true.
On the other hand, we also have $(-\Delta) = 4 \overline \partial ^* \overline \partial $
where $\overline \partial = (1/2)(\partial_x + \iota \partial_y)$ 
is the Cauchy-Riemann operator. In this case, \eqref{eq.Korn.Laplace} 
fails on holomorphic functions on $\Omega$ from the Sobolev class $H^1 (\Omega)$. 

The Lam\'e system from the elasticity theory
$
{\mathcal L} = 
-\mu\Delta - (\mu + \lambda)\mathop{\nabla}\operatorname{div}
$ 
(with positive constants $\mu$ and $\lambda$) 
is known to be strongly elliptic (and even a generalized Laplacian) 
and satisfy \eqref{eq.Korn}, see the 
original paper by Korn \cite{Korn08} or classical book \cite{Fi72} by 
G. Fichera or elsewhere. However, similarly to the Laplace operator 
on complex-valued functions, this system admit a factorization that 
does not fit for the Korn inequality, 
% and the Neumann problem, corresponding to this factorization is not coercive, 
see \cite{PeSh15}.

Actually, inequality \eqref{eq.Korn} 
provides that the following Neumann Problem has the Fredholm property 
on the scale of the Sobolev spaces for the domain $\Omega\Subset X$, see, 
for instance, \cite{Simanca1987}.

\begin{problem} \label{pr.Neu}
 Given $k$-vector function
$ u_1 \in [H^{s-3/2} (\partial \Omega)]^k $ %defined on $D$ and $\partial D$ respectively
%$\oplus_{j=0}^{m-1} u_j \in 
%\oplus_{j=m}^{2m-1}  [H^{s-j-1/2} (\partial D)]^k$ 
find, if possible a $k$-vector function $u \in [H^{s} (\Omega)]^k $ such that 
\begin{equation*} %\label{eq.Neumann.Laplacian}
\left\{ \begin{array}{lll}
L  u =0 & {\rm in} & \Omega,\\
\partial_{L,\nu} u  =u_1
%\oplus_{j=m}^{m-1} B_j u= \oplus_{j=0}^{m-1} u_j 
& {\rm on} & \partial \Omega,\\
\end{array}
\right.
\end{equation*}
where 
$
\partial_{L,\nu} = \sum_{i,j=1}^n \nu_j L_{i,j}    \partial _i 
%\Big(-\sum_{i,j=1}^n \nu_j L_{i,j}    \partial _i + \sum_{i=1}^n  \nu_i L_i \Big) 
$ 
is the "co-normal derivative"{} with respect to $\partial \Omega$ related to $L$ 
with the exterior unit normal vector $\nu (x)= (\nu_1 (x), \dots \nu_n(x))$ 
 to the surface $\partial \Omega$ at the point $x$. 
\end{problem}
In the case of the usual Laplace operator in ${\mathbb R}^{2}$ and $A=\nabla$ Problem 
\eqref{pr.Neu} is the classical Neumann problem where $B_1$ coincides 
with the normal derivative while the operator 
$A=2\overline \partial $ corresponds to the $\overline \partial$-Neumann problem 
with $\sum_{i,j=1}^n L_{i,j}  \nu_j (x) \partial _i = (\nu_1 + \iota \nu_2)\overline 
\partial$. The  $\overline \partial$-Neumann problem is not Fredholm on the scale 
of the Sobolev spaces because 
the space of solutions of its homogeneous version is precisely infinite-dimensional 
space of holomorphic functions of the corresponding Sobolev class over $D$. 
The normal solvabilty of such non-Fredholm problems 
can be established by the method  of sub-elliptic estimates, see for instance, 
\cite{Kohn79}, \cite{PeSh15}, \cite{ShTaDU}; however it may affect the behaviour of boundary 
integrals in the Green formula \eqref{eq.Green.M.two}. 

In our particular case, for the Dirichlet pairs 
\begin{equation} \label{eq.B.L}
B=\big(B_0=I_k,B_1= \sum_{j=1}^n \nu_j L_j u -\partial_{L,\nu}\big),
\, \tilde B =(\tilde B_0=I_k, \tilde B_1= \partial_{L^*,\nu})
\end{equation}
 (the first) Green formula \eqref{eq.Green.M.B} reads as follows:
\begin{equation*}
%\label{eq.Green.M.L}
\int_{\partial \Omega} \big( (\tilde B_1 v)^* u - v ^* B_1 v
\big) d\sigma = \int_{\Omega} \big( v^*  L u - 
(L^* v)^* u  \big) dx,
\end{equation*}
where 
$$
L^* v= - \sum_{i,j=1}^n \partial_i (L^*_{i,j}  \partial_j v) -
\sum_{i=1}^n \partial _i( L^*_{i}   v) + L^*_0 v.
$$
For the generalized Laplacians $A^*A$ 
the boundary operators \eqref{eq.B.L} are the following:
\begin{equation} \label{eq.B.Laplace}
B_0 = \tilde B_0 =I_k, \, B_1 = \tilde B_1 =\sum_{j=1}^n A_j^* \nu_j A.
\end{equation}

As before, we assume that both the operators $L$ and $L^*$ 
possess the Unique continuation property (US) on $X$. Then 
the potential ${\mathcal T}_S (u_0,u_1)$ corresponding to $L$ and 
the Dirichlet pairs \eqref{eq.B.L} can be present as the sum 
\begin{equation} \label{eq.T.m=1}
{\mathcal T}^{(B)}_S (u_0,u_1) (x) = {\mathcal V}^{(B)}_{S} (u_1) 
+ {\mathcal W}^{(B)}_{S} (u_0)
\end{equation}
where  
\begin{equation} \label{eq.V.m=1}
{\mathcal V}^{(B)}_{S} ( u_1) = \int_S 
\Phi (x,y)  u_1 (y)
%u_0  (y) \tilde B_1 (y)  \varphi_M (x,y)-  u_1 (y) %\tilde B_0 (y) \varphi_M (x,y)
d\sigma (y) 
\end{equation}
is the single layer potential and
\begin{equation} \label{eq.W.m=1}
{\mathcal W}^{(B)}_{S} (u_0) = -\int_S 
(\partial_{L^*,\nu (y)} \Phi^* (x,y) )^* u_0 (y)
%u_0  (y) \tilde B_1 (y)  \varphi_M (x,y)-  u_1 (y) %\tilde B_0 (y) \varphi_M (x,y)
%\Big) 
d\sigma (y)
\end{equation}
is the double layer potential. Note that 
in some book the classical double layer 
potential related to the Laplace operator is defined without the sign "-"; this 
affects on the sign in formula \eqref{eq.T.m=1}, see, for instance, \cite{rjasanow2007fast}.

 We will be concentrated 
on domains with Lipschitz boundaries as the most important case for applications.

Let $\Omega $ be a relatively compact domains in $X$ with Lipschitz 
boundary. According to \cite[Theorem 1]{Cost88} the operators 
\begin{equation} \label{eq.VB}
{\mathcal V}^{(B)}_{\partial \Omega} :
[H^{s-3/2} (\partial \Omega)]^k \to [H^{s} 
( \Omega)]^k
\end{equation}
\begin{equation} \label{eq.WB}
{\mathcal W}^{(B)}_{\partial \Omega} :
[H^{s-1/2} (\partial \Omega)]^k \to [H^{s} 
( \Omega)]^k
\end{equation}
%are bounded  for $s\in (\frac{1}{2}, \frac{3}{2})$  and the operator 
\begin{equation} \label{eq.integral}
B_0 {\mathcal V}^{(B)}_{\partial \Omega} :
[H^{s-3/2} (\partial \Omega)]^k \to [H^{s-1/2} 
(\partial \Omega)]^k
\end{equation}
\begin{equation} \label{eq.integral.W}
B_0 {\mathcal W}^{(B)}_{\partial \Omega} :
[H^{s-1/2} (\partial \Omega)]^k \to [H^{s-1/2} 
(\partial \Omega)]^k
\end{equation} 
are bounded for $s\in (\frac{1}{2}, \frac{3}{2})$ 
if  $L$ is a second order strongly elliptic operator satisfying 
\eqref{eq.Dir.Hadamard}, \eqref{eq.Korn}. Moreover, taking in account Remark \ref{r.m=1} we conclude that for any $s\in (\frac{1}{2}, \frac{3}{2})$  and $u\in [H^s (\Omega)]^k \cap 
S_L (\Omega)$ the (representation) Green formula 
\begin{equation}
\label{eq.Green.Lip}
\chi_{\Omega} u = {\mathcal W}^{(B)}_{\partial \Omega} B_0 u + 
{\mathcal V}^{(B)}_{\partial \Omega} B_1 u
\end{equation}
is still valid for the Lipschitz domain $\Omega$ under the assumptions above because 
the the following boundary operator can be treated as continuous 
\begin{equation}
\label{eq.B1.Lip}
B_1  :[H^s (\Omega)]^k \cap S_L (\Omega) \to [H^{s-3/2} (\partial \Omega)]^k,
\end{equation}
see \cite[Lemma 3.7]{Cost88}.

This allows to 
extend Theorem \ref{t.Cauchy.M} to the case of Lipschitz boundaries for $m=1$.

\begin{theorem} \label{t.Cauchy.M.Lip} 
 Let $L$ be a second order strongly  elliptic operator such that both 
$L$ and $L^*$ satisfy (US) on $X$ and satisfy  
\eqref{eq.Dir.Hadamard}, \eqref{eq.Korn}. Let also $s\in (\frac{1}{2}, \frac{3}{2})$ and let 
$\Omega_0$ be a bounded Lipschitz domain in $X$. 
%and the matrix $M$ have real analytic entries and satisfy \eqref{eq.M.pos}. 
%If $\partial D \setminus S$ has at least one interior point in the relative 
%topology then 
If $\Omega_1 \setminus  \Omega_0$ has no compact components in $\Omega_1$ then 
Problem \ref{pr.Cauchy.M} is densely solvable and 
%If $S$ is a relatively open subset of $\partial D$ with a smooth boundary $\partial S$ then
%Cauchy problem \ref{pr.Cauchy.M} 
it has no more than one solution. It 
is solvable if and only if there is 
a function ${\mathcal F}\in [H^s (\Omega_1)]^k\cap S_{L}(\Omega_1)$ 
%satisfying $\Delta_M {\mathcal F} = 0$ in  $G$
and such that 
$$
{\mathcal F}= 
({\mathcal W}^{(B)}_{\partial \Omega_0} (u_0))^+ + 
({\mathcal V}^{(B)}_{\partial \Omega_0} (u_1))^+
\mbox{ in } \Omega_0.
$$
Besides, the solution $u$, if exists, is given by the following formula:
\begin{equation} \label{eq.sol.Cauchy.Lip}
u  = ({\mathcal W}^{(B)}_{\partial \Omega_0} (u_0))^- + 
({\mathcal V}^{(B)}_{\partial \Omega_0} (u_1))^-
- {\mathcal F} \mbox{ in } \Omega_1 \setminus \overline \Omega_0.
\end{equation}
\end{theorem}

\begin{proof}  The proof actually follows the same way as 
the proofs of Theorem \ref{t.Cauchy.M}. Namely, if the solution $u$ exists then 
$$
{\mathcal F}= ({\mathcal W}^{(B)}_{\partial \Omega_0} (u_0)) + 
({\mathcal V}^{(B)}_{\partial \Omega_0} (u_1)) -\chi_{\Omega_1 \setminus \overline 
\Omega_0} u = -({\mathcal W}^{(B)}_{\partial \Omega_1} (B_0 u)) - 
({\mathcal V}^{(B)}_{\partial \Omega_1} (B_1 u))
$$
that is obviously belongs to $[H^s (\Omega_1)]^k\cap S_{L}(\Omega_1)$.

Back, if there is 
a function ${\mathcal F}\in [H^s (\Omega_1)]^k\cap S_{L}(\Omega_1)$  
we may invoke the jump formula \eqref{eq.jump} for potentials. In this particular case  
they are still valid for Lipschitz surfaces and they have the following  form:
\begin{equation} \label{eq.jump.V2} 
B_1 ({\mathcal V}^{(B)}_{\partial \Omega_0} (u_1)
)^- - B_1 ({
\mathcal V}^{(B)}_{\partial \Omega_0} (u_1)
)^+ =u_{1} \mbox{ on } \partial \Omega, \, 
\end{equation}
\begin{equation*} %\label{eq.jump.V1} 
B_0({\mathcal V}^{(B)}_{\partial \Omega_0})(u_1))^- 
-  B_0({\mathcal V}^{(B)}_{\partial \Omega_0} ( u_1))^+ =0 \mbox{ on } \partial \Omega_0,
\end{equation*}
\begin{equation*} %\label{eq.jump.W1}
B_0({\mathcal W}^{(B)}_{\partial \Omega_0} (u_0)
)^- -  B_0({\mathcal W}^{(B)}_{\partial \Omega_0}
(u_0) )^+ =u_0 \mbox{ on } \partial \Omega_0, 
\end{equation*}
\begin{equation*} %\label{eq.jump.W2}
 B_1 ({\mathcal W}^{(B)}_{\partial \Omega_0}
(u_0) )^- -  B_1  ({\mathcal W}^{(B)}_{\partial \Omega_0}
(u_0) )^+ =0  \mbox{ on } \partial \Omega_0
\end{equation*}
see \cite[Lemma 4.1]{Cost88}. According to them the function $u$ defined by 
\eqref{eq.sol.Cauchy.Lip} is the solution to Problem \ref{pr.Cauchy.M}.
\end{proof}

Now we may extend Corollary \ref{c.ext.1} to the Lipschitz domains for $m=1$.

\begin{corollary} \label{c.ext.1.Lip}
 Let $L$ be a second order strongly  elliptic operator such that both 
$L$ and $L^*$ satisfy (US) on $X$ and satisfy  
\eqref{eq.Dir.Hadamard}, \eqref{eq.Korn}. Let also $s\in (\frac{1}{2}, \frac{3}{2})$ and let 
$\Omega_j$ be  bounded Lipschitz domains in $X$. 
If $\Omega_1 \setminus  \Omega_0$ has no compact components in $\Omega_1$ then 
Problem \ref{pr.ext.1} is densely solvable and 
%If $S$ is a relatively open subset of $\partial D$ with a smooth boundary $\partial S$ then
%Cauchy problem \ref{pr.Cauchy.M} 
it has no more than one solution. It 
is solvable if and only if  Problem \ref{pr.Cauchy.M} 
is solvable for the data $u_0 =B_0V$ and $u_1 =B_1 V$ %and 
%$\oplus_{j=0}^{m-1} u_{m+j}= \oplus_{j=1}^m \tilde B_{m-1-j} A V$ 
on $\partial \Omega_0$. 
%%{\mathcal P} (\oplus_{j=1}^m B_j u)$  
\end{corollary}
 
\begin{proof} Actually, all the arguments are the same as in the proof 
of Corollary \ref{c.ext.1}: we need to correct then in the part 
related to the dense solvability. But according to \cite[Exercise 1.4.10]{Tark36} 
the bounded Lipschitz domain $\Omega_0$ has the so-called \textit{strong cone property}
and then $[H^s (\Omega_1)]^k\cap S_L (\Omega_1)$ is dense in 
$[H^s (\Omega_0)]^k\cap S_L (\Omega_0)$ because of  
\cite[Corollary 8.4.2]{Tark36}.
\end{proof}

Thus we arrive at the end of our discussion of Problems 
\ref{pr.Cauchy.M} \ref{pr.ext.1}, and \ref{pr.ext.2} for $m=1$ and Lipschitz boundaries.

\begin{corollary} \label{c.ext.2.Lip}
Let $L$ be a second order strongly  elliptic operator such that both 
$L$ and $L^*$ satisfy (US) on $X$ and satisfy  
\eqref{eq.Dir.Hadamard}, \eqref{eq.Korn}. Let also $s\in (\frac{1}{2}, \frac{3}{2})$ and let 
$\Omega_j$ be  bounded Lipschitz domains in $X$. 
%and the matrix $M$ have real analytic entries and satisfy \eqref{eq.M.pos}. 
%If $\partial D \setminus S$ has at least one interior point in the relative 
%topology then 
If 
$\Omega_1 \setminus  \Omega_0$ has no compact components in $\Omega_1$ then 
Problem \ref{pr.ext.2} is densely solvable and 
%If $S$ is a relatively open subset of $\partial D$ with a smooth boundary $\partial S$ then
%Cauchy problem \eqref{pr.Cauchy.M} 
it has no more than one solution. Moreover, the following 
statements are equivalent: 
\begin{itemize}
\item
Problem \ref{pr.ext.2}  is solvable with the %Dirichlet 
data 
$v_0 \in  [H^{s-1/2} (\partial \Omega_0)]^k $;
\item
Problem \ref{pr.ext.1}  is solvable with $V= {\mathcal P}_{\Omega_0} 
(v_0)$;
\item 
 Problem \ref{pr.Cauchy.M} 
is solvable for the Cauchy data $u_0$, $u_1$  where $u_0=v_0$ 
 and   $u_1 = B_1 {\mathcal P}_{\Omega_0} (v_0)$ on $\partial \Omega_0$.  
\end{itemize}
\end{corollary}

\begin{corollary} \label{c.well.ext.1.Lip} 
Let $L$ be a second order strongly  elliptic operator such that both 
$L$ and $L^*$ satisfy (US) on $X$ and satisfy  
\eqref{eq.Dir.Hadamard}, \eqref{eq.Korn}. Let also $s\in (\frac{1}{2}, \frac{3}{2})$ and let 
$\Omega_j$ be  bounded Lipschitz domains in $X$.
 If $\Omega_1 \setminus  \Omega_0$ has no compact components in $\Omega_1$ 
then  Problems \ref{pr.Cauchy.M},  \ref{pr.ext.1} 
are conditionally well-posed in the following senses:
\begin{enumerate}
\item
if a sequence $\{U_i\} \subset [H^s_{L,\gamma} (\Omega_1)]^k$
converges to zero in  $[H^s ( \Omega_0)]^k$ then 
$\{U_i\} $ converges to zero in the local space $[H_{\rm loc}^{s} ( \Omega_1 )]^k$ 
and weakly converges to zero in $[H^{s} ( \Omega_1 )]^k$;
\item
if for a sequence $\{u_i\} \subset [H^s_{L,\gamma}(\Omega_1\setminus \overline\Omega_0)]^k$
the sequences $\{B_j u_i\}$ converge to zero in  $[H^{s-j-1/2} (\partial \Omega_0)]^k$ 
%and  sequences $\{\tilde B_{m-j-1} A u_i\}$ converge to zero in  $[H^{s-2m+j+1/2}
% (\partial \Omega_0)]^k$ 
for each $j$, $0\leq j\leq 2m-1$, then 
$\{u_i\} $ converges to zero in the space $[H_{\rm loc}^{s} ( \Omega_1 
\setminus \Omega_0)]^k$.
\end{enumerate}
If, moreover, $L$  satisfies \eqref{eq.Dir.Hadamard} 
then Problem \ref{pr.ext.2}  is conditionally well-posed in the following sense: 
\begin{enumerate}
\item[(3)]
if for a sequence $\{U_i\} \subset [H^s_{L,\gamma}(\Omega_1)]^k$ the sequences $\{B_j U_i\}$ 
converge to zero in  $[H^{s-j-1/2} (\partial \Omega_0)]^k$ for all $j$, $0\leq j \leq m-1$, 
then $\{U_i\} $ converges to zero in the space $[H_{\rm loc}^{s} ( \Omega_1 )]^k$.
\end{enumerate}
%and weakly in $[H^s ( \Omega_1)]^k$.
\end{corollary}

Let us formulate the main result of this section related 
to the problem of representation of solutions to $L$ 
as a single layer potential.
The following  expectable 
 and rather known statement 
  follows almost immediately  from \cite[Theorem 2 and 3]{Cost88}.

\begin{theorem} \label{t.single.layer}
Let $\Omega $ be a relatively compact domains in $X$ with Lipschitz 
boundary and $L$ be a second order strongly elliptic operator satisfying 
\eqref{eq.Dir.Hadamard}, \eqref{eq.Korn} and 
\begin{equation} \label{eq.Dir.Hadamard.Omega}
[H^{1}_0 (\Omega)]^k
\cap S_{L} (\Omega)= [H^1_0 (\Omega)]^k
\cap S_{L^*} (\Omega) =0. 
\end{equation} 
%\begin{equation} \label{eq.Dir.Hadamard.Omega.compl}
%[H^{m}_0 (D\setminus \overline \Omega)]^k
%\cap S_{L} (D\setminus \overline  \Omega)= [H^{m}_0 (D\setminus \overline  \Omega)]^k
%\cap S_{L^*} (D\setminus \overline  \Omega) =0. 
%\end{equation} 
 If operator \eqref{eq.integral} is injective 
then for any 
$u\in [H^{s} (\Omega)]^k\cap S_L (\Omega)$, $s\in [1,\frac{3}{2}) $, %(\frac{1}{2},\frac{3}{2})
there is a unique function $u_1 \in [H^{s-3/2} (\partial \Omega)]^k$ such that 
\begin{equation} \label{eq.single}
u = {\mathcal V}^{(B)}_{\partial \Omega} (u_1) \mbox{ in } \Omega 
\end{equation}
related to $L$ and Dirichlet pairs \eqref{eq.B.L}. 
Moreover, if 
$\partial \Omega$ is a surface  of class $C^\infty$, %$s \in \mathbb N$, $s\geq 2$, then the 
statements is still true for all $s \in \mathbb N$. 
\end{theorem}

\begin{proof} We begin the discussion for domains Lipschitz 
boundaries.  First, we note that  the operator \eqref{eq.integral}
is strongly elliptic in the following sense (\cite[Theorem 2]{Cost88}):
%$$
%\|w\|_{[H^{-1/2} (\partial \Omega)]^k } \leq 
%\|{\mathcal V}^{(B)}_{\partial \Omega} (w)\|_{[H^{1/2} (\partial \Omega)]^k} + 
%$$
%for each $w \in [H^{-1/2} (\partial \Omega)]^k $.
there is a compact operator $K: [H^{-1/2} (\partial \Omega)]^k \to [H^{1/2} 
( \partial \Omega)]^k$ and a positive constant $C$ such that
\begin{equation} \label{eq.ell.est}
\Re{\langle (B_0 {\mathcal V}^{(B)}_{\partial \Omega} + K)w,w \rangle }
\geq C \,\|w\|^2_{[H^{-1/2} (\partial \Omega)]^k}
\mbox{ for all } w \in [H^{-1/2} (\partial \Omega)]^k
\end{equation}
where the brackets $\langle \cdot , \cdot\rangle$ denote the natural duality pairing 
between the Sobolev space $[H^{s} (\partial \Omega)]^k$ and its dual.

As it is well-known, estimate \eqref{eq.ell.est} is equivalent to 
the fact that the operator $B_0 {\mathcal V}^{(B)}_{\partial \Omega}$ 
admits a parametrix on the space $[H^{-1/2} (\partial \Omega)]^k$, i.e. there are compact 
operators $ {\mathcal K }^{(l)} : [H^{-1/2} (\partial \Omega)]^k \to [H^{-1/2} 
( \partial \Omega)]^k$, 
$ {\mathcal K }^{(r)} : [H^{1/2} (\partial \Omega)]^k \to [H^{1/2} 
( \partial \Omega)]^k$   and a bounded 
operator $ {\mathcal Q} : [H^{1/2} (\partial \Omega)]^k \to [H^{-1/2} 
( \partial \Omega)]^k$ such that 
\begin{equation*} %\label{eq.param.l}
Q (B_0 {\mathcal V}^{(B)}_{\partial \Omega}) + 
{\mathcal K }^{(l)}  = I \mbox{ on } 
[H^{-1/2} (\partial \Omega)]^k, 
\end{equation*}
\begin{equation*} %\label{eq.param.r}
(B_0 {\mathcal V}^{(B)}_{\partial \Omega}) Q  + 
{\mathcal K }^{(r)}  = I \mbox{ on } 
[H^{1/2} (\partial \Omega)]^k.
\end{equation*}

Second, the regularity results (\cite[Theorem 3]{Cost88})
allow us to define a parametrix for the operator 
$B_0 {\mathcal V}^{(B)}_{\partial \Omega}$  on the scale 
$[H^{s-3/2} (\partial \Omega)]^k \to [H^{s} ( \Omega)]^k$, $s\in (\frac{1}{2}, \frac{3}{2})$, 
i.e. there are compact 
operators $ {\mathcal K }_s ^{(l)}: [H^{s-3/2} (\partial \Omega)]^k \to [H^{s-3/2} 
( \partial \Omega)]^k$ , 
$ {\mathcal K }_s ^{(r)}: [H^{s-1/2} (\partial \Omega)]^k \to [H^{s-1/2} 
( \partial \Omega)]^k$   and a bounded 
operator $ {\mathcal Q} : [H^{s-1/2} (\partial \Omega)]^k \to [H^{s-3/2} 
( \partial \Omega)]^k$ such that 
\begin{equation} \label{eq.param.s.left}
 {\mathcal Q}_s (B_0 {\mathcal V}^{(B)}_{\partial \Omega}) + 
{\mathcal K }^{(l)}_s  = I \mbox{ on } 
[H^{s-3/2} (\partial \Omega)]^k,
\end{equation}
\begin{equation} \label{eq.param.s.right}
 (B_0 {\mathcal V}^{(B)}_{\partial \Omega})  {\mathcal Q}_s + 
{\mathcal K }_s ^{(r)} = I \mbox{ on } 
[H^{s-1/2} (\partial \Omega)]^k.
\end{equation}

Next,  formula \eqref{eq.Dir.Hadamard.Omega} and Corollary \ref{c.Dirichlet.Poisson}
imply that a function 
$u \in S_L (\Omega) \cap [H^{s} (\Omega)]^k$ 
satisfy \eqref{eq.single}
with a function $u_1\in [H^{s-3/2} (\partial \Omega)]^k$ if and only if
\begin{equation} \label{eq.single.boundary}
B_0 u = B_0 {\mathcal V}^{(B)}_{\partial \Omega} (u_1) \mbox{ on } \partial \Omega .
\end{equation}

According to \eqref{eq.param.s.left}, \eqref{eq.param.s.right}, on the space 
$[H^{s-3/2} (\partial \Omega)]^k$ we have
$$
B_0 {\mathcal V}^{(B)}_{\partial \Omega} =
B_0 {\mathcal V}^{(B)}_{\partial \Omega} \Big(
 {\mathcal Q}_s (B_0 {\mathcal V}^{(B)}_{\partial \Omega}) + 
{\mathcal K }^{(l)}_s  \Big)=
$$
$$
I+
B_0 {\mathcal V}^{(B)}_{\partial \Omega} {\mathcal K }^{(l)}_s + 
{\mathcal K }^{(r)}_s B_0 {\mathcal V}^{(B)}_{\partial \Omega}.
$$
Thus, since compositions of compact and bounded operators are compact, we see
that  equation \eqref{eq.single.boundary} reduces to a Fredholm operator equation. 

Hence, it follows from Fredholm alternative that the operator 
\eqref{eq.integral} is continuously invertible if and only if it is injective 
for $s \in [1, \frac{3}{2})$. Thus, as \eqref{eq.integral} is injective, 
we conclude that equation \eqref{eq.single.boundary} 
is always uniquely solvable in $[H^{s-3/2} (\partial \Omega)]^k$ with $s \in [1,
\frac{3}{2})$.

Finally,  if 
$\partial \Omega$ is a surface  of class $C^\infty$ %, $s \in \mathbb N$, $s\geq 2$, 
then the trace operators \eqref{eq.trace.B_j} are continuous for all $s >1/2$.  
The potential operator \eqref{eq.VB} is bounded 
this $s>1/2$ too, because of theorems related the pseudo differential operators satisfying the 
so-called transmission property, see, for instance, 
\cite[\S 2.3.2.5 ]{RS82} or \cite[\S 2.4]{Tark36}. 

Hence operator 
\eqref{eq.integral} is continuous and 
the statement follows from the standard scheme of improvement 
of the regularity for solutions to elliptic pseudo-differential equations,  
that was to be proved. Actually, the statement of the theorem follows 
from more general results on strongly elliptic 
boundary operators by M. Costabel and W. L. Wendland \cite[Theorem 3.7]{CoWe86}.
\end{proof}

Let us indicate important cases where the operator \eqref{eq.integral} is 
injective.

\begin{theorem} \label{t.single.layer.Green}
Let $\Omega \Subset D$ be a relatively compact domains in $X$ with Lipschitz 
boundary and $L$ be a second order strongly elliptic operator satisfying 
\eqref{eq.Dir.Hadamard}, \eqref{eq.Korn}, \eqref{eq.Dir.Hadamard.Omega} and 
\begin{equation} \label{eq.Dir.Hadamard.Omega.compl}
[H^1_0 (D\setminus \overline \Omega)]^k
\cap S_{L} (D\setminus \overline  \Omega)= [H^1
_0 (D\setminus \overline  \Omega)]^k
\cap S_{L^*} (D\setminus \overline  \Omega) =0. 
\end{equation} 
 If $\Phi= {\mathcal G}_D$ is the 
Green function of the Dirichlet Problem \ref{pr.Dir} for $L$ in $D$ then for any 
$u\in [H^{s} (\Omega)]^k\cap S_L (\Omega)$, $s\in [1,\frac{3}{2}) $, %(\frac{1}{2},\frac{3}{2})
there is a unique function $u_1 \in [H^{s-3/2} 
(\partial \Omega)]^k$ such that \eqref{eq.single} is fulfilled. 
Moreover, if 
$\partial \Omega$ is a surface  of class $C^\infty$, then the 
statements is still true for this $s \in \mathbb N$.
\end{theorem}

\begin{proof} Let us establish that operator \eqref{eq.integral} is 
injective under the hypothesis of this corollary. 

Let a function $v \in [H^{s-3/2} (\partial \Omega)]^k$ belongs to  
the kernel of the operator \eqref{eq.integral}. In order to prove that $v$ equals to 
zero we invoke jump formulas \eqref{eq.jump.V2}

As $L {\mathcal V}^{(B)}_{\partial \Omega} (v) = 0$ in $D \setminus \partial \Omega$ 
and $B_0 {\mathcal V}^{(B)}_{\partial \Omega} (v)_{|\Omega} = 0$ on $\partial \Omega$,  
we see that  $({\mathcal V}^{(B)}_{\partial \Omega} (v))_{|\Omega} \equiv 0$ in $\Omega$ 
because \eqref{eq.Dir.Hadamard.Omega} and Corollary \ref{c.Dirichlet.Poisson}. 
Similarly, since $\Phi$ is the Green function of the 
Dirichlet Problem \ref{pr.Dir} for $L$ in $D$ we see that 
$({\mathcal V}^{(B)}_{\partial \Omega} (v)_{|D \setminus \overline \Omega} $ 
is the solution to the homogeneous Dirichlet problem for $L$ in 
$D \setminus \overline \Omega$. Then \eqref{eq.Dir.Hadamard.Omega.compl} and Corollary 
\ref{c.Dirichlet.Poisson} yield
$({\mathcal V}^{(B)}_{\partial \Omega} (v))_{|
D \setminus \overline \Omega} \equiv 0$ in $D \setminus \overline \Omega$.
Now jump formula \eqref{eq.jump.V2} (applied to $\partial \Omega$ instead of 
$\partial \Omega_0$) implies that $v=0$ on $\partial \Omega$, 
i.e. the operator \eqref{eq.integral} is injective.
Finally, the Fredholm alternative mean that equation \eqref{eq.single.boundary} 
is always uniquely solvable in $[H^{s-3/2} (\partial \Omega)]^k$ with $s \in [1,
\frac{3}{2})$,  that was to be proved.
\end{proof}

For the generalized Laplacians the statement of Theorem 
\ref{t.single.layer.Green} can be 
formulated much shorter.

\begin{corollary} \label{c.single.layer}
Let $\Omega \Subset D$ be a relatively compact domain in $X$ with Lipschitz 
boundary and $A$ be an operator with injective principal symbol satisfying (US)
on $X$. If  $L= A^*A$  satisfies \eqref{eq.Korn.Laplace} and 
$\Phi= {\mathcal G}_D$ is the 
Green function of the Dirichlet Problem \ref{pr.Dir} for $A^*A$ in $D$ then for any 
$u\in [H^{s} (\Omega)]^k\cap S_{A^*A} (\Omega)$, $s\in [1, \frac{3}{2}) $, 
there is a function $u_1 \in [H^{s-3/2} 
(\partial \Omega)]^k$ such that  \eqref{eq.single} holds true. 
Moreover, if 
$\partial \Omega$ is a surface  of class $C^\infty$,  then the 
statements is still true for this $s \in \mathbb N$.
\end{corollary}

\begin{proof}
If $L=A^*A$ with a first order operator $A$ having 
injective principal symbol and satisfying (US)
on $X$ then \eqref{eq.Dir.Hadamard}, 
\eqref{eq.Dir.Hadamard.Omega} and \eqref{eq.Dir.Hadamard.Omega.compl}  hold true, 
see Corollary \ref{c.Dirichlet.Hodge}. As 
\eqref{eq.Korn.Laplace} is equivalent to \eqref{eq.Korn} for $L=A^*A$, 
we conclude that the statement of the corollary follows from 
Theorem \ref{t.single.layer}.
\end{proof}

As it is well known, in general, the solvability of the corresponding 
boundary singular integral equations on $\partial \Omega $ is closely related to solvability 
of some "interior"{}  and "exterior"{} boundary problems 
for the elliptic operator, see, for instance, \cite{Vla} for classical solutions  
or \cite{FJR78}, \cite{CoWe86} for the Sobolev type spaces. 
In our particular case it is an 
"exterior"{} Dirichlet problem, as 
Theorem \ref{t.single.layer.Green} and Corollary \ref{c.single.layer} confirm. 
%
%Consider the following "exterior"{} Dirichlet problem for $L$.
%\begin{problem} \label{pr.Dir.ext}
%Let $s\in [1, \frac{3}{2}) $. 
%Given function $u_0 \in 
%  [H^{s-1/2} (\partial \Omega)]^k$ find, if possible a 
%function $u \in [H^{s} (B (0,R) \setminus \overline \Omega)]^k$ such that 
%\begin{equation} \label{eq.Dirichlet.Laplacian}
%\left\{ \begin{array}{lll}
%L u =0 & {\rm in} & {\mathbb R}^n \setminus \overline \Omega,\\
%u= u_0 & {\rm on} & \partial \Omega,\\
%\lim_{|x|\to + \infty} u(x) = 0.
%\end{array}
%\right.
%\end{equation}
%\end{problem}
However an "exterior"{} Dirichlet problem for general strongly elliptic 
operator $L$ on ${\mathbb R}^n\setminus \Omega$ can 
be rather complicated even in the case of scalar operators, see, 
for instance, \cite{MeSe60}.

At the end of this section, let us consider the classical case 
of second order elliptic operators with constant coefficients in ${\mathbb R}^n$. 
In this case $L$ admits a bilateral fundamental solution of the convolution type, 
say $\Phi(x-y)$, $x,y \in {\mathbb R}^n$, $x\ne y$, and hence it is 
natural to consider potentials constructed with the use of this kernel.
We will consider a rather particular situation.  
where $n\geq 2$ and $A = \sum_{j=1}^n A_j \partial_j$ is a homogeneous 
first order  operator with constant coefficients in ${\mathbb R}^n$ 
having injective principal symbol. 
Then its Laplacian $A^*A$ 
admits the fundamental solution of the convolution type in the form 
\begin{equation} \label{eq.Phi.homo.0}
\Phi (x) = a\Big( \frac{x}{|x|}\Big)|x|^{2-n} + b(x) \ln{|x|}
\end{equation}
where  $a(\zeta)$ is a $(k\times k)$-matrix of real analytic functions in a neighbourhood of 
the sphere $\{|\zeta|=1\}$ and $b(x)$ is a $(k\times k)$-matrix of polynomials $a_{p,q} (\zeta)$
of order $(2-n)$, see, for instance, \cite[\S 2.2.2]{Tark97}. This means that 
$b \equiv 0$ and 
\begin{equation} \label{eq.Phi.homo}
\Phi (x) = a\Big( \frac{x}{|x|}\Big)|x|^{2-n} \mbox{ for } n\geq 3, 
\end{equation}
and $b$ is a matrix with constant entries for $n=2$.

\begin{corollary} \label{c.single.layer.const}
Let $\Omega $ be a relatively compact domain in $X$ with Lipschitz 
boundary in ${\mathbb R}^n$, $n\geq 3$ and $A$ be a first order 
homogeneous operator with constant coefficients 
having injective principal symbol and satisfying 
 \eqref{eq.Korn}. 
 If $\Phi$ is given by \eqref{eq.Phi.homo}
%the  bilateral fundamental solution of the convolution type
%for the generalized Laplacian $A^*A$ in ${\mathbb R}^n$ such that 
%\begin{equation} \label{eq.decay}
%\lim_{|x| \to +\infty}\max_{i,j}|\Phi_{i,j} (x)| \, 
%\max_{q,p} |(A \Phi)_{q,p} (x)| \, |x|^{n-1} =0
%\end{equation}
%\begin{equation} \label{eq.decay.2}
%\partial_j \Phi (x) \leq  \frac{c_j}{|x|^{n-1}} \mbox{ for all } x \in {\mathbb R}^n 
%\setminus \{0\}, \, \, 1\leq j \leq n, 
%\end{equation}
%with positive constant $c_j$, $0\leq j\leq n$ independent on $x$, 
then for any 
$u\in [H^{s} (\Omega)]^k\cap S_L (\Omega)$, $s\in [1,\frac{3}{2}) $, 
there is a unique function $u_1 \in [H^{s-3/2} 
(\partial \Omega)]^k$ such that \eqref{eq.single} holds true. 
Moreover, if 
$\partial \Omega$ is a surface  of class $C^\infty$,  then the 
statements is still true for all $s \in \mathbb N$.
\end{corollary}

\begin{proof} Let $A $ 
is a first order homogeneous operator with constant coefficients 
having injective principal symbol satisfying 
 \eqref{eq.Korn}. Then $A$ satisfies (US) 
condition in ${\mathbb R}^n$.

Thus, according to Theorem \ref{t.single.layer}, the remaining 
fact to establish is the injectivity of the corresponding operator 
\eqref{eq.integral}.

Let a function $v \in [H^{s-3/2} (\partial \Omega)]^k$ belongs to the kernel of the operator 
\eqref{eq.integral}. Again, the Uniqueness of the Dirichlet problem for $A^*A$ in 
$\Omega$ means that ${\mathcal V}^{(B)}_{\partial \Omega} (v)$ equals identically to 
zero in $\Omega$. However, for $s\geq 1$, using integration 
by parts,  we easily obtain 
$$
0=\langle A^*A ({\mathcal V}^{(B)}_{\partial \Omega} (v)) , {\mathcal V}^{(B)}_{\partial \Omega} (v) \rangle_{{\mathbb R}^n \setminus \overline \Omega} = 
$$
$$
\lim_{R\to + \infty}\Big( \|A {\mathcal V}^{(B)}_{\partial \Omega} (v) \|^2_{L^2 (B(0,R) 
\setminus \overline \Omega)} - 
\int_{|x|=R} 
\sum_{j=1}^n ({\mathcal V}^{(B)}_{\partial \Omega} (v) (x))^* \frac{x_j}{|x|}A^*_j
A{\mathcal V}^{(B)}_{\partial \Omega} (v) (x) d\sigma(x) \Big). 
$$
Clearly, as $\Omega$ is bounded, then using the particular type 
\eqref{eq.Phi.homo} of the kernel $\Phi (x-y)$, we see that 
there is a number $R_1>0$ such that 
$$
|{\mathcal V}^{(B)}_{\partial \Omega} (v) (x)|\leq 
\|v\|_{[H^{s-3/2} (\partial \Omega)]^k} |x|^{2-n}
\max_{p,q} \|a_{p,q}\Big(\frac{x-y}{|x-y|}\Big) \left|\frac{x-y}{|x|} \right|^{2-n} \|_{[H^{3/2-s}_y (\partial \Omega)]^k}
\leq  
$$
\begin{equation}\label{eq.V.infty}
C_\Omega ^{(1)} \|v\|_{[H^{s-3/2} (\partial \Omega)]^k} |x|^{2-n}
%\max_{i,j} |\Phi_{i,j} (x)|
\end{equation}
with a positive constant $C^{(1)}_\Omega$ independent on $x$ with $|x|\geq R_1$ and $v$.

On the other hand, 
$$
\partial_{j} \Phi (x-y) = (\partial_j a) \Big(\frac{x-y}{|x-y|}\Big) 
\frac{\partial}{\partial y_j} \Big(\frac{x-y}{|x-y|}\Big) |x-y|^{2-n} + 
a \Big(\frac{x-y}{|x-y|}\Big) \partial_j |x-y|^{2-n},
$$
$$
|\partial_{j} \Phi (x-y) | \leq C  |x|^{1-n} \left| \frac{x-y}{|x|} \right|^{1-n}
$$
with a constant $C$ independent on $x$ and $y$. 
Then, using the homogeneity of the operator $A$, we see that 
there is a number $R_2>0$ such that 
%$$
\begin{equation}\label{eq.AV.infty}
|A{\mathcal V}^{(B)}_{\partial \Omega} (v) (x)|\leq 
%\|v\|_{[H^{s-3/2} (\partial \Omega)]^k} |x|^{1-n} 
%\max_{p,q} \|(A_x\Phi)_{p,q} \Big(\frac{x-y}{|x|}\Big)\|_{[H^{3/2-s}_y (\partial \Omega)]^k}
%\leq  
%$$
%\begin{equation}\label{eq.AV.infty}
C^{(2)}_\Omega  \|v\|_{[H^{s-3/2} (\partial \Omega)]^k} |x|^{1-n}
%\max_{p,q} |(A\Phi)_{p,q} (x)|
\end{equation}
with a positive constant $C^{(2)}_\Omega$ independent on $x$ 
 with $|x|\geq R_2$ and $v$.

Now estimates \eqref{eq.V.infty}, \eqref{eq.AV.infty} imply for 
$$
\left|
\int_{|x|=R} 
\sum_{j=1}^n ({\mathcal V}^{(B)}_{\partial \Omega} (v) (x))^* \frac{x_j}{|x|}A^*_j
A{\mathcal V}^{(B)}_{\partial \Omega} (v) (x) d\sigma(x)
\right| \leq R^{2-n}
$$
for $R\geq \max (R_1,R_2) $ and then
$$
0 =\lim_{R\to + \infty} \|A {\mathcal V}^{(B)}_{\partial \Omega} (v) \|^2_{L^2 (B(0,R) 
\setminus \overline \Omega)}
= \|A {\mathcal V}^{(B)}_{\partial \Omega} (v) \|^2_{L^2 ({\mathbb R}^n 
\setminus \overline \Omega)},
$$
i.e. $A {\mathcal V}^{(B)}_{\partial \Omega} (v) =0$ in ${\mathbb R}^n 
\setminus \overline \Omega$. 

Then jump formula \eqref{eq.jump.V2} with the boundary operators 
\eqref{eq.B.Laplace} yields
$$
v= B_1 ({\mathcal V}^{(B)}_{\partial \Omega} (v))^- - 
B_1 ({\mathcal V}^{(B)}_{\partial \Omega} (v))^+ =0 \mbox{ on } \partial \Omega,
$$
i.e. operator \eqref{eq.integral} is injective, that was to be proved. 
\end{proof}

\begin{example} Let $n\geq 2$. The typical example of  
a homogeneous first order  operator $A$ with constant coefficients in ${\mathbb R}^n$ 
having injective principal symbol is the gradient operator $A=\nabla$.
Of course, its Laplacian $A^*A = -\Delta$ satisfies  \eqref{eq.Korn.Laplace}.
Its fundamental solution of type \eqref{eq.Phi.homo.0} is given by 
$$
\Phi(x) = \varphi_n (x) = \left\{ \begin{array}{ll}
\frac{|x|^{2-n}}{(n-2)\sigma_n}, & n\geq 3,\\
\frac{1}{2\pi}\ln{|x|}, & n=2, \\
\end{array}
\right.
$$
where $\sigma_n$ is the square of the unit sphere in ${\mathbb R}^n$. 
Thus, we see that $\varphi_2$ does not have a sufficient decay at the infinity in order 
to prove Corollary \ref{c.single.layer.const} for $n=2$ by the 
present arguments (estimate \eqref{eq.V.infty} is 
not fulfilled!). However, for $n=2$ injectivity of the corresponding potential 
with logarithmic kernel  may be established by more sophisticated methods under specific 
assumptions on the geometry of the curve $\partial \Omega$, 
see, for instance, \cite{HsW77}, \cite{YS88}. 
\end{example}

\begin{example} Let $n\geq 2$. The typical example of  
a homogeneous second order strongly elliptic 
matrix operator with constant coefficients in ${\mathbb R}^n$ is the 
Lam\'e $(n\times n)$-system from the elasticity theory 
${\mathcal L} = -\mu\Delta - (\mu + \lambda)\mathop{\nabla}\operatorname{div}$ 
with the Lam\'e constants $\mu$ and $\lambda$ such that $\mu>0$, $2\mu+\lambda>0$. It  
is known to satisfy \eqref{eq.Korn}, see the 
original paper by Korn \cite{Korn08} or classical book \cite{Fi72}. 
Actually, it can be factorized as ${\mathcal L}={\mathcal A}^* {\mathcal A}  $ 
a first order homogeneous $(k\times n)$-matrix  operator 
${\mathcal A} = 
\sum_{j=1}^m {\mathcal A}_j \partial _j   $. Of course, there are many 
such operators ${\mathcal A} $. To introduce three of them we denote by 
 $M_1 \otimes M_2$ the Kronecker product of matrices $M_1$ and $M_2$, 
 by $\mbox{rot}_m$  we denote  $\Big(\frac{(m^2-m)}{2}\times 
m\Big)$-matrix operator with the lines   $\vec{e}_i \frac{\partial}{\partial x_j}-\vec{e}_j 
\frac{\partial}{\partial x_i }$,  $1\leq i<j \leq m$,  representing the vorticity (or the 
standard rotation operator for $m=2$, $m=3$), and by  ${\mathbb D}_m$ we denote  
$\Big(\frac{(m^2+m)}{2}\times m\Big)$-matrix operator with the lines  $\sqrt{2} \vec{e}_i 
\frac{\partial}{\partial x_i}$, $1\leq i \leq m$, and $\vec{e}_i \frac{\partial}{\partial 
x_j}+\vec{e}_j \frac{\partial}{\partial x_i }$  with $1\leq i<j \leq m$,  
representing the deformation (the strain). The we set: 
\begin{equation} \label{eq.factor.1}
{\mathcal A}^{(1)} = 
\left(\begin{array}{lll} \sqrt{\mu } \ {\mathbb D}_m \\ %\sqrt{\mu } \ \mbox{antirot}_m \\ 
\sqrt{\lambda } \mbox{div}_m  \\ \end{array} \right), \, 
{\mathcal A}^{(2)}  = \left(\begin{array}{lll} \sqrt{\mu } \ \nabla_m \otimes I_m \\  \sqrt{\mu 
+\lambda } \ \mbox{div}_m , \end{array} \right), \, {\mathcal A}^{(3)} = 
\left(\begin{array}{lll} \sqrt{\mu } \ \mbox{rot}_m \\ \sqrt{2\mu +\lambda } \mbox{div}_m  \\
\end{array} \right), 
\end{equation}  
here $\lambda \geq 0$,  $k_1=(m^2+m)/2 + 1$ for the first operator,   $(\mu+\lambda)\geq 0$, 
$k_2=m^2+1$  for the second operator, and $(2\mu+\lambda)>0$, 
$k_3=(m^2-m)/2 + 1$  for the third operator.  
Factorized as $({\mathcal A}^{(1)})^* {\mathcal A}^{(1)}$ or 
$({\mathcal A}^{(2)})^* {\mathcal A}^{(2)}$, 
it satisfies \eqref{eq.Korn.Laplace}; see, for instance, \cite[Examples 3 and 5]{PeSh15} or \cite{Fi72}.
Factorization $({\mathcal A}^{(3)})^* {\mathcal A}^{(3)}$ does not admit 
\eqref{eq.Korn.Laplace} because the Neumann problem \ref{pr.Neu}
is not coercive in this case, \cite[Example 4]{PeSh15}.
 
Its fundamental solution of type \eqref{eq.Phi.homo.0} is given by 
$
\Phi_n (x) = \left( \Phi^{(n)}_{ij} (x) \right)_{i,j=1,2,...,n}
%\left\{ \begin{array}{ll}
%\frac{|x|^{2-n}}{(n-2)\sigma_n}, & n\geq 3,\\
%\frac{1}{2\pi}\ln{|x|}, & n=2, \\
%\end{array}
%\right.
$ 
 with
components
$$\Phi^{(n)}_{ij} (x) =
   \frac{1}{2 \mu (\lambda + 2 \mu)}
   \left( \delta_{ij}\, (\lambda + 3 \mu )  g(x) -
          (\lambda + \mu)\, x_{j}\, \frac{\partial}{\partial x_{i}} \varphi_n (x)
   \right)
   \, \, \, (i,j = 1,2,...,n), $$
where $\delta_{ij}$ is the Kronecker delta and $\varphi_n (x)$ is the standard fundamental 
solution to the Laplace operator  in ${\mathbb R}^n$. 
Thus, we see that $\Phi_2$ does not have a sufficient decay at the infinity in order 
to prove Corollary \ref{c.single.layer.const} for $n=2$ by the 
present arguments (estimate \eqref{eq.V.infty} is 
not fulfilled!). 
%However, for $n=2$ injectivity of the corresponding potential 
%with logarithmic kernel  may be established by more sophisticated methods under specific 
%assumptions on \textcolor{red}{the geometry of the} curve $\partial \Omega$, 
%see, for instance, \cite{HsW77}, \cite{YS88}. 
%\textcolor{red}{Do you know a reference to 
%the fact that Corollary \ref{c.single.layer.const} is not true in this case ?}
\end{example}

\begin{example} Let $a$ be  a complex non-zero number and 
$$A = \left( 
\begin{array}{ll}
\nabla\\
a
\end{array}\right) ;
$$
this operator is not homogeneous. 
 Then $A^*A = |a|^2- \Delta$ is the Helmholtz operator  in ${\mathbb R}^3 $ 
admitting the fundamental solutions
$$
\Phi_- (x) = \frac{e^{-|a|\,|x|}}{4\pi |x|}, \, \Phi_+(x) = \frac{e^{|a|\,|x|}}{4\pi |x|}
$$
  and \eqref{eq.Korn.Laplace} is obviously fulfilled. Clearly, %\eqref{eq.decay}
we may repeat the arguments from the proof of Corollary \ref{c.single.layer.const}
and prove that the representation 
via single layer potential is true for solutions to the Helmholtz equation 
if we choose  $\Phi_- (x)$ for the potential. 
However estimates \eqref{eq.V.infty}, \eqref{eq.AV.infty} are not  true for $\Phi_+(x) $ 
and the arguments fail for the corresponding potential. 

Note that the Helmholtz operator  in ${\mathbb R}^3 $  may also have the form 
$(-|a|^2-\Delta)$. However, in this case its standard 
fundamental solutions are 
$$
\Phi_- (x) = \frac{e^{-\iota \, |a|\,|x|}}{4\pi |x|}, \, \Phi_+(x) = 
\frac{e^{\iota |a|\,|x|}}{4\pi |x|};
$$
again estimates \eqref{eq.V.infty}, \eqref{eq.AV.infty} are not  true for 
the potentials corresponding  to both $\Phi_- (x)$ and $\Phi_+(x) $. Hence 
the arguments in the proof of Corollary \ref{c.single.layer.const} fail. 
\end{example}

%\section{On the discrete approximation of solutions}
%\label{S.4}

 Finally, let us show how the theorems  on the 
representation by the single layer potential and approximation theorems 
for solutions to elliptic systems help to clarify the question 
on the so-called \textit{discrete approximation}.

\begin{corollary} \label{c.discrete}
Let $s,s'\in \mathbb N$ and let 
$\Omega$ be a relatively compact domain in $X$ with $C^\infty$-smooth 
boundary. Let also $L$ be a second order strongly elliptic operator satisfying 
\eqref{eq.Dir.Hadamard}, \eqref{eq.Korn}, \eqref{eq.Dir.Hadamard.Omega} and 
the injectiviy condition for operator \eqref{eq.integral}. If $\Omega'$ is 
a relatively compact domain  in $\Omega$ with Lipschitz boundary 
such that set $\Omega\setminus \Omega'$ has no compact components then for any 
the set of all single layer potentials of type \eqref{eq.single} with densities 
from the space  $[H^{s-3/2} (\partial \Omega)]^k$ are dense in the space 
$[H^{s'} (\Omega')]^k \cap S_{L} (\Omega')$.
\end{corollary}

\begin{proof} 
Approximation Theorems for solutions to elliptic systems,
 see, for instance, \cite[Theorems 5.1.11, 5.1.13, 8.2.2]{Tark36} imply 
that the space $[H^{s} (\Omega)]^k \cap S_{L} (\Omega)$ is dense in 
$[H^{s'} (\Omega')]^k \cap S_{L} (\Omega')$. Hence the statement follows from 
Theorem \ref{t.single.layer}.
\end{proof}

The following corollary is just a specification of similar statements 
for various solutions to elliptic systems, see the pioneer  result 
\cite{R1885} by C. Runge for holomorphic functions where the Cauchy kernel was used,  
or \cite[Theorem 4.2.1]{Tark97} for differential operators admitting left fundamental 
solutions.

\begin{corollary} \label{c.discrete.2}
Let $s'\in \mathbb N$ and let 
$\Omega$ be a relatively compact domain in $X$ with $C^\infty$-smooth 
boundary. Let also $L$ be a second order strongly elliptic operator satisfying 
\eqref{eq.Dir.Hadamard}, \eqref{eq.Korn}, \eqref{eq.Dir.Hadamard.Omega} and 
the injectiviy condition for operator \eqref{eq.integral}. If $\{z_j\}_{j\in \mathbb N}$ is an everywhere dense 
set on $\partial \Omega$ and $\Omega'$ is 
a relatively compact domain  in $\Omega$ with Lipschitz boundary 
such that set $\Omega\setminus \Omega'$ has no compact components then for any 
$u \in [H^{s'} (\Omega')]^k \cap S_{L} (\Omega')$ and any $\varepsilon >0$
there are numbers $M (u,\varepsilon) \in \mathbb N$ and 
$k$-vectors $\{c_j (u,\varepsilon)\}_{j=1}^{M (u,\varepsilon)}$ 
such that
$$
\left\|u (x) - \sum_{j=1}^{M (u,\varepsilon)}   \Phi (x,z_j) 
c_j (u,\varepsilon) \right\|_{[H^{s'} (\Omega')]^k }<\varepsilon.
$$
\end{corollary}

\begin{proof} Fix a positive number $\varepsilon$.
First, take domain $\Omega''$ such that $\Omega' \Subset \Omega'' \Subset \Omega$. 
Then, by the a priori estimates for elliptic systems, see, for instance, 
\cite{GiTru83}, we know that 
there is a number $c(s',\partial \Omega', \partial \Omega'')>0$ dependent on the distance between $\partial \Omega'$ and 
$\partial \Omega''$  such that 
\begin{equation} \label{eq.apriori}
\|u\|_{H^{s'} (\Omega')} \leq c(s',
\partial \Omega',\partial \Omega'') \|u\|_{[C(\overline{\Omega''})]^k}.
\end{equation}
for all $u \in S_L (\Omega'') \cap [C(\overline{\Omega''}]^k$.

Second, using Corollary \ref{c.discrete} for $s>\frac{n+2}{2}$, we pick a density
a density $v_\varepsilon \in [H^{s-3/2} (\partial \Omega)]^k$ 
such that
\begin{equation} \label{eq.discr.appr.1}
\left\|u  -  {\mathcal V}^{(B)}_{\partial \Omega} (v_\varepsilon)
\right\|_{[H^{s'} (\Omega')]^k }<\varepsilon/2.
\end{equation}
By the Sobolev embedding theorems, see, for instance, 
\cite[Ch.~4, Theorem~4.12]{Ad03}), the density $v_\varepsilon$ is actually continuous 
on $\partial \Omega$. In particular, 
the potential ${\mathcal V}^{(B)}_{\partial \Omega} (v_\varepsilon)$ is 
a Riemann integral depending on the parameter $x \in \Omega''$ as $\Omega''\Subset \Omega$.

Note that the  set $K=  \overline {\Omega''} \times \partial \Omega$ is a compact in 
${\mathbb R}^n \times {\mathbb R}^n$. By Cantor's theorem, any continuous 
function on $K$ is uniformly continuous and hence for any  $w \in C(K)$  there is 
a number $\delta_\varepsilon$ such that 
\begin{equation} \label{eq.Cantor}
|w(x',y') - w(x'',y''| <\frac{\varepsilon }{2 c(s',\partial \Omega',\partial \Omega'')  \sigma (\partial \Omega)} 
\end{equation}
if $x',x'' \in \overline \Omega''$, $y', y'' \in \partial \Omega$
and $|x'-x''|^2 +  |y'-y''|^2$, where $\sigma (\partial \Omega)$ 
is the square ($(n-1)$-Jordan measure) of the hypersurface $\partial \Omega$.

As the function $\Phi^* (x, y) v_\varepsilon (y)$ 
is continuous on $K$, it is Riemann integrable
over $\partial \Omega$. Then for any partition $P = \{ P_i\}_{i=1}^N$ of 
$\partial \Omega$ by measurable sets $P_i \subset \partial \Omega$ we have  
for all $x \in \Omega''$:
$$
 {\mathcal V}^{(B)}_{\partial \Omega} (v_\varepsilon) (x) 
= \sum_{i=1}^N  {\mathcal V}^{(B)}_{P_i} (v_\varepsilon) (x) 
$$
Then, by the integral mean value theorem, 
there are points $y_{i,x} \in P_i$ such that 
\begin{equation} \label{eq.mean}
 {\mathcal V}^{(B)}_{\partial \Omega} (v_\varepsilon) (x) 
= \sum_{i=1}^N  \Phi (x, y_{i,x})  v_\varepsilon (y_{i,x})  
% =\sum_{i=1}^N  \sigma (P_i)  \Phi (x, y_{i,x}) \textcolor{red}{v_\varepsilon} (y_{i,x}).   
\end{equation}
Next, we choose $N=N(\varepsilon)$ and $\{ P_i = P_i (\varepsilon)\}_{i=1}^N $ 
to be relatively open sets on $\partial \Omega$ with 
piece-wise smooth boundaries and such that the diameter of each set 
is less than the number $\delta_\varepsilon$ related to \eqref{eq.Cantor}. 
Since the set of points $\{z_j\}$ is every where dense in $\partial \Omega$ 
we conclude that each $P_i$ contains at least one point $ z_{i(\varepsilon}) \in \{z_j\} $. 
Hence \eqref{eq.Cantor} yields for each $x \in \overline \Omega''$:
\begin{equation} \label{eq.discr.appr.2}
\left|\sum_{i=1}^{N(\varepsilon)} 
\Big( \Phi (x, y_{i,x}) v_\varepsilon (y_{i,x}) - 
 \Phi (x,  z_{i(\varepsilon}))  v_\varepsilon ( z_{i(\varepsilon}))\Big)
\sigma (P_i(\varepsilon)) 
 \right|
\leq 
\end{equation}
\begin{equation*}
\sum_{i=1}^{N(\varepsilon)}   
\frac{\varepsilon \sigma (P_i (\varepsilon))}{2 c(s',\partial \Omega',\partial \Omega'') \sigma (\partial \Omega)}
= \frac{\varepsilon}{2 c(s',\partial \Omega',\partial \Omega'')}.
\end{equation*}
Finally, combining 
\eqref{eq.apriori}, \eqref{eq.discr.appr.1},   \eqref{eq.mean}, \eqref{eq.discr.appr.2} 
we conclude that the statement of the corollary holds true with 
$M (\varepsilon) = \max_{1\leq i \leq N(\varepsilon)} i(\varepsilon)$ and 
the vectors 
$$
c_j (u,\varepsilon) = 
\left\{
\begin{array}{lll}
\sigma (P_i(\varepsilon)) v_\varepsilon ( z_{i(\varepsilon})) & 
\rm{ if } & j= i(\varepsilon), \\
0 & \rm{ if } & j\ne  i(\varepsilon).
\end{array}
\right.
$$
This finishes the proof.
\end{proof}
Note that the described scheme   allows 
us to use the set of points $\{ z_j \}_{j=1}^{M(\varepsilon )}$ 
after refining the partition $\{P_i \}_{j=1}^{N(\varepsilon )}$ 
for a new positive number $\varepsilon'<\varepsilon$, adding new points 
from the set $\{ z_j \}_{j=1}^ \infty$ related to the 
new elements of the  refined  partition that do not contain 
the elements of the set  $\{ z_j \}_{j=1}^{M(\varepsilon )}$. 

\section{Some methods for solving the exterior extension problems for strongly elliptic 
operators of the second order}
\label{S.4}

In this section we consider some methods for constructing solutions to Problems 
\ref{eq:problem1}, \ref{eq:problem2}, \ref{eq:problem3} (or, more precisely, 
Problems \ref{pr.Cauchy.M}, \ref{pr.ext.1}, \ref{pr.ext.2}, respectively) which are based 
on the use of fundamental solutions of the corresponding elliptic equations.  More 
specifically, we will focus on the so-called indirect method of boundary integral equations 
in terms of the single layer and on the method of fundamental solutions.  We also considered 
the “extension approach” for approximation of the solution to the Dirichlet Problem 
\ref{eq:problem4} (more precisely, Problem \ref{pr.Dir} for $m=1$).  

In this section we assume that $m=1$,  $L$ is a second order strongly elliptic operator with 
smooth coefficients such that $L$ and $L^{*}$ satisfy (US) 
property in $X$ and requirements 
\eqref{eq.Korn}, \eqref{eq.Dir.Hadamard.Omega} and the injectivity condition for the single 
layer operator \eqref{eq.integral}  in the respective domains. 
We assume also that  the hypotheses of Corollary \ref{c.Dirichlet.Poisson} 
and Corollary \ref{c.well.ext.1} or Corollary \ref{c.well.ext.1.Lip} are fulfilled.

As in the previous sections, we assume that  $\Omega_0$ is a relatively 
compact domain in $X$ and $\Omega_1$ is a bounded domain in $X$ such that $\bar{\Omega}_0 
\subset \Omega_1$ and $\Omega_1 \setminus \Omega_0$ has no compact components; solutions to 
the boundary value problems belong to $[H^{s}(\Omega_1)]^k \bigcap S_L (\Omega_1)$. 
In this section we will use the following requirements to the domain boundaries:

%\textcolor{blue}{ 
1) the boundaries $\partial \Omega_0$ and $\partial \Omega_1$ belongs to $C^{s}$-class of smoothness  if $s\in \mathbb N$, $s\geq 2$ 
or 2) both boundaries $\partial \Omega_0$ and $\partial \Omega_1$ 
are Lipschitz ones if  $s \in [1, 3/2)$. 
%}

We denote the restrictions of $B_0u$ and  $B_1u$ onto $\partial \Omega_i$ as 
 $B_{0,\partial \Omega_i}$ and  $B_{1,\partial \Omega_1}$ resepectively. We also denote  $B_{0,\partial \Omega_i}u$ as $u_{0i}$
and  $B_{1,\partial \Omega_i}u$ as  $u_{1i}$.

We begin our consideration with the direct boundary integral equations method for solving 
Problem \ref{pr.ext.2}.

The \textbf{indirect boundary integral equation method} in terms of the single layer consists 
in computing the solution to Problem \ref{pr.ext.2}  
as a potential of the single layer given on $\partial \Omega_1$:
\begin{equation}
\label{eq.4.8}
u(x) =  {\mathcal V} _{\partial \Omega_1} v (z), x\in \Omega_1, z\in \partial\Omega_1, 
\end{equation} 
where $v$ is a density of the single layer.  The unknown single layer density $v$ can be found as a solution to the operator equation: 

\begin{equation}
\label{eq.4.10}
{\mathcal V}_{\partial \Omega_1} v(z) = f(y),  z\in \partial\Omega_1,  y\in \partial\Omega_0,
\end{equation} 
where $f=B_0u$ on $\partial\Omega_0$ is the boundary datum of Problem \ref{pr.ext.2}.

%$f=B_{0,{\partial\Omega_0}}u$

\begin{corollary} \label{t.4.3} 
If a solution $u \in [H^{s}(\Omega_1)]^k \bigcap S_L(\Omega_1)$  to Problem \ref{pr.ext.2} 
exists, then the solution $v \in [H^{s-3/2} (\partial \Omega_1)]^k$  to 
operator equation \eqref{eq.4.10} exists and it is unique. Equation \eqref{eq.4.10} is densely 
solvable. Any solution $u \in [H^{s}(\partial \Omega_1)]^k  \bigcap S_L(\Omega_1)$  to 
Problem \ref{pr.ext.2}  %\ref{eq:problem3} 
can be presented in form \eqref{eq.4.8} where 
$v \in [H^{s-3/2} (\partial \Omega_1)]^k$ is the solution to equation 
\eqref{eq.4.10}. 
\end{corollary} 

\begin{proof} 
The statements follow immediately from  Corollary \ref{c.well.ext.1.Lip} 
and Theorem \ref{t.single.layer}.
\end{proof}

\begin{corollary} \label{t.4.4}
Let $\{ v^{(i)} \} \subset  \Big[ H^{s-\frac{3}{2}}  (
\partial \Omega_1)  \Big]^k$ be a bounded sequence. 

If the sequence 
%${G_{01} v^{(i)}}$ 
$f^{(i)}$ of the boundary datum of Problem \ref{pr.ext.2}
strongly convergence to zero in 
$\big[H^{s-1/2}(\partial \Omega_0) \big]^k $ then:
\begin{itemize}
\item the sequence $\{u^{(i)}\}$ given by formula \eqref{eq.4.8} with $v=v^{(i)}$ strongly 
converges to zero in $[H^s (\Omega_0)]^k \bigcap S_L (\Omega_0)$, 
weakly converges to zero  in $[H^s (\Omega_1)]^k \cap S_L (\Omega_1)$  
 and it converges to zero 
in the local space $[H_{\text{loc}}^s (\Omega_1 )]^k$; 
\item the sequence 
%$G_{11}v^{(i)}$ 
$B_{0,{\partial\Omega_1}}u$
weakly converges to zero in 
$\big[H^{s-1/2} (\partial \Omega_1)\big]^k$  
 and the sequence $\{v^{(i)}\}$ is weakly converges to zero in 
$\big[H^{s-3/2} (\partial \Omega_1)\big]^k$.
\end{itemize}
\end{corollary} 

\begin{proof}
First, note that 
%$G_{01}v^{(i)} = B_{0} u^{(i)}$ 
$B_0{\mathcal V}_{\partial \Omega_1} v^{(i)} = B_0u^{(i)}=f^{(i)}$
and according to Theorem 
\ref{t.single.layer}, each $u^{(i)}$ can be presented by formula \eqref{eq.4.8}. 
As the sequence $\{ v^{(i)} \} \subset  \Big[ H^{s-\frac{3}{2}}  
(\partial \Omega_1)  \Big]^k$ is bounded, the sequence  
$\{u^{(i)}\}$ is bounded in the space $\big[H^{s} (\Omega_1)\big]^k$ because 
of the boundedness of the single layer potential 
$
{\mathcal V}^{(B)}_{\partial \Omega_1} :
[H^{s-3/2} (\partial \Omega)_1]^k \to [H^{s} 
( \Omega_1)]^k
$, 
see \cite[\S 2.3.2.5 ]{RS82} or \cite[\S 2.4]{Tark36} for $s\in \mathbb N$ or 
\cite[Theorem 1]{Cost88}  for $s \in (\frac{1}{2}, \frac{3}{2})$, i.e. 
$\{u^{(i)}\} \subset \big[H^{s}_{L,\gamma} (\Omega_1)\big]^k $ with 
some non-negative number $\gamma$. 

Then Corollary \ref{c.well.ext.1.Lip} 
states that if $\{B_0 u^{(i)} \}$ convergence to zero in 
$\big[H^{s-1/2} (\partial \Omega_0)\big]^k$ then the Problem \ref{pr.ext.2} 
solution sequence $\{u^{(i)}\}$ strongly converges to zero in the space 
$[H^s (\Omega_0) 
\bigcap S_L (\Omega_0)]^k$, 
weakly converges to zero  in $[H^s (\Omega_1)]^k \cap S_L (\Omega_1)$ and it converges to 
zero in the local space $[H_{\text{loc}}^s (\Omega_1)]^k$.

Second, the weak convergence of $\{u^{(i)}\}$ in 
$[H^s (\Omega_1 )]^k\cap S_L (\Omega_1)$ to zero implies the 
weak convergence of
% $\left\{ G_{01} v^{(i)} \right\} = \left\{ B_{0} u^{(i)} \right\}$ 
$B_0 u^{(i)}$ in $\big[H^{s-1/2} (\partial \Omega_1)\big]^k$ 
due to the continuity of the trace 
operator $B_0$. As it was shown in the proof of Theorem \ref{t.single.layer}, the operator 
$$
{\mathcal V}_{\partial \Omega_1}: \big[H^{s-3/2} (\partial \Omega_1)\big]^k \rightarrow 
\big[H^{s-1/2} (\partial \Omega_1)\big]^k
$$ 
is continuously invertible under our 
assumptions. This fact provides a weak convergence of the sequence $\{v^{(i)}\}$ in 
$\big[H^{s-3/2} (\partial \Omega_1)\big]^k$.
\end{proof}

The \textbf{method of fundamental solutions (MFS) for solving Problem 
%\ref{eq:problem3} 
\ref{pr.ext.2} } 
consists in setting an additional larger bounded domain $\Omega_2 \subset X$ such that 
$\bar{\Omega}_1 \subset \Omega_2$ and representation of the solution to Problem \ref{pr.ext.2} 
%\ref{eq:problem3} 
in $\Omega_1$ by a weighted sum of fundamental solutions whose 
singularities are located on the boundary $\partial \Omega_2$ of 
domain $\Omega_2$:
%\begin{equation*}
\begin{equation} 
\label{eq.4.11}
u(x) \approx \sum_{j=1}^{N}  %\cdot 
\Phi (x, z_j) c_j,
\end{equation} 
where $x \in \Omega_1$, $\{z_j\}^{N}_{j=1}$ is a set of isolated points of 
$\partial \Omega_2$ and $c_j$ are some vectors from ${\mathbb R}^k$.

We consider the case when $\Omega_2$ has a $C^{\infty}$-smooth boundary $\partial \Omega_2$  
and $\Omega_2 \setminus \Omega_1$ has no compact components.
%
%\textcolor{red} {As in the previous case, we assume that the boundaries $\partial \Omega_0$ 
%and $\partial \Omega_1$ are Lipschitz surfaces.}
%
%We consider Problem \ref{pr.ext.2} solution in  $[H^s (\Omega_1)]^k \bigcap 
%S_L (\Omega_1)$ where $s \in \mathbb N$. 
% 
Under these assumptions the 
possibility of approximation of the solution to Problem \ref{pr.ext.2} 
%\ref{eq:problem3} 
by this method and the conditional stability of this approximation is justified by the following results. 

\begin{corollary} \label{t.4.5} %Theorem 5. 
If $\{ z_j \}_{j \in \mathbb N}$ is an everywhere dense set of points on $\partial \Omega_2$ 
then for any boundary datum $f \in \big[H^{s-1/2} (\partial \Omega_0)\big]^k$ of 
Problem \ref{pr.ext.2} %\ref{eq:problem3} 
and for any solutions $[H^s (\Omega_1)]^k \bigcap 
S_L (\Omega_1)$ to Problem \ref{pr.ext.2}  % \ref{eq:problem3} 
(if it exists) and any $\varepsilon > 0$ there 
are numbers $M(u,\varepsilon)$ and the weight coefficients vector 
$\{c_j (u,\varepsilon)\}_{j=1}^{M(u,\varepsilon)}$ such that 
%\begin{equation*}
\begin{equation*}
\left\| u - \sum_{j=1}^{M}  %\cdot 
\Phi(x,z_j) c_j \right\|_{[H^s (\Omega_1)]^k} < \varepsilon, \,\,
%x \in \Omega_1, %\{z_j\} \subset \partial \Omega_2, 
%\end{equation*}
%\begin{equation*}
\left\| f - \sum_{j=1}^{M}  %\cdot 
\Phi(x,z_j) c_j\right\|_{[H^{(s-1/2)} (\partial \Omega_1)]^k} < \varepsilon.
%x \in \Omega_1 . %, \{z_j\} \subset \partial \Omega_2,
\end{equation*} 
\end{corollary} 
 
\begin{proof} The statements follow immediately from Corollary 
\ref{c.discrete.2} and the continuity of the trace operator $B_0$. 
\end{proof}

\begin{corollary} \label{t.4.6} %Theorem 6.
Let $\{ z_j \}_{j \in \mathbb N}\subset \partial \Omega_2$. 
If a sequence %of series 
$\left\{ u^{(i)} = \sum_{j=1}^{\infty} 
 %\cdot 
\Phi(x,z_j)c_j^{(i)} \right\}^{i \in \mathbb N}_{ x \in \partial \Omega_1 }
$
is bounded in $\big[ H^{s}  (\Omega_1) \big]^k$ and the sequence of series 
$
\left\{ u_0^{(i)} = B_0  u^{(i)} %\sum_{j=1}^{\infty} 
 %\cdot 
%\Phi(x,z_j) \textcolor{red}{c_j^{(i)}} 
\right\}^{i \in \mathbb N}_{ x \in \partial \Omega_0 }$ strongly 
converges to zero in $\big[ H^{s-1/2}  (\partial \Omega_0) \big]^k$ then the 
sequence of series $\{ u^{(i)}\}$ 
%$$\Big\{ \sum_{j=1}^{\infty}  \Phi (x, z_j) \textcolor{red}{c_j^{(i)}} \Big\}^{i \in 
%\mathbb N}_{ x \in \Omega_1}
%$$ 
strongly converges to zero in $[H^s(\Omega_0)]^{k} 
\bigcap S_L (\Omega_0)$, weakly converges to zero in the space 
$[H (\Omega_1 %\setminus \Omega_0
)]^k$, it converges to zero  also in the local space 
$[H_{\text{loc}}^s(\Omega_1 %\setminus \Omega_0
) ]^k $ and the sequence of series 
$\left\{ B_0 u^{(i)}
%\Big\{ B_0 \sum_{j=1}^{\infty} \Phi (x, z_j) \textcolor{red}{c_j^{(i)} } 
\right\}^{i \in \mathbb N}
_{ x \in \Omega_1}
$ 
weakly converges to zero in $\big[H^{s-1/2} (\partial \Omega_1)\big]^k$.
\end{corollary}

\begin{proof} The statements follow from 
%Corollary \ref{c.well.ext.1} 
Corollary \ref{c.well.ext.1.Lip} 
%\ref{t.4.N} 
and Corollary 
\ref{c.discrete.2} and the continuity of the trace operator $B_0$.
\end{proof}

Next, we will discuss the “extension” method for approximate solving the 
\textbf{interior Dirichlet Problem \ref{eq:problem4}, or, more precisely, 
Problem \ref{pr.Dir} } for the operator
$L$ in $\Omega_0$ with the Dirichlet condition $u_{00}$ on $\partial \Omega_0$. 

The method consists in introducing a virtual embracing boundary $\partial \Omega_v$ such that 
$\bar{\Omega}_0 \subset \Omega_v$  and solving Problem \ref{pr.ext.2}  in $\Omega_v$  with 
the same condition $u_{00}$ on $\partial \Omega_0$. A restriction of the solution to Problem 
\ref{pr.ext.2}  on $\Omega_0$ is considered as an approximation of the Dirichlet problem 
solution $u$ in $\Omega_0$.
 
To solve Problem \ref{pr.ext.2}, taking into account the approximation of the Problem 
\ref{pr.Dir} solution, one can use the integral equation method  discussed above, as well as 
the method of fundamental solutions. For this purpose, the virtual surface $\partial \Omega_v$ 
can be considered as $\partial \Omega_1$ for the direct and indirect methods of integral 
equations and as $\partial \Omega_2$ for the MSF. In the last case there is no need to 
explicitly set the intermediate boundary $\partial \Omega_1$.

Corollaries \ref{t.4.3} - \ref{t.4.6} justify the applicability these methods for solving  
Problem \ref{pr.Dir} in $[H^s (\Omega_0)]^k \bigcap S_L (\Omega_0)$ subject to the 
appropriate assumptions about the smoothness of the boundary $\partial \Omega_0$ and the 
prescribed boundary $\partial \Omega_v$. 
Namely, for the integral equation methods
%\textcolor{blue}{ 
1) the boundaries $\partial \Omega_0$ and $\partial \Omega_v$ belong to $C^{s}$-class of smoothness  if $s\in \mathbb N$, $s\geq 2$
or 2) both boundaries $\partial \Omega_0$ and $\partial \Omega_v$ 
are Lipschitz ones if  $s \in [1, 3/2)$. 
%} 
%\textcolor{blue}{
For the MFS $\partial \Omega_0$ is assumed to be of the same class of smoothness and $\partial
 \Omega_v$ has  $C^{\infty}$  smoothness. 
%}

In particular,  Corollaries \ref{t.4.3} - \ref{t.4.6}  guarantee an arbitrarily accurate 
approximation of the solution of Problem \ref{pr.Dir} and the stability of its solution to 
small perturbations of the boundary condition in the respective norms. 

The same technique of reduction to the problem \ref{pr.ext.2} can also be used for solving 
Problem \ref{pr.ext.1} of “analytical” continuation from the domain 
$\partial \Omega_0$ to the large domain $\partial \Omega_1$.  

Actually,  knowing the datum $V \in [H^{s} (\Omega_0)]^k \bigcap S_L(
\Omega_0)$ implies knowing its trace $B_0 u = \big[H^{s-1/2} 
(\partial \Omega_0)\big]^k$ on $\partial \Omega_0$. Therefore, as we showed above 
solving Problem \ref{pr.ext.1}
can be reduced to solving Problem \ref{pr.ext.2}  
with the boundary datum  $B_0 V$ on $\partial \Omega_0$ 
using the  integral equation methods or MFS. All the statements of 
Corollaries \ref{t.4.3}-\ref{t.4.6} for applying these methods to solve Problem 
\ref{pr.ext.1}  are valid under the above assumption.

%\textcolor{red} {(To present this in the form of a theorem?)}
 
Finally, we %will 
consider some ways to compute a solution $u \in [H^{s} (\Omega_1 \setminus \Omega_0)]^k 
\bigcap S_L(\Omega_1 \setminus \Omega_0)$ to the Cauchy Problrm \ref{pr.Cauchy.M} with a pair 
of the Cauchy data $\{u_{00}, u_{10}\}$,  $ u_{00} \in \big[H^{s-1/2} (\partial \Omega_0)
\big]^k$, $ u_{10} \in \big[H^{s-3/2} (\partial \Omega_0)\big]^k$.

These methods are based on the application of Theorem \ref {t.Cauchy.M.Lip}, more precisely, 
formula  \eqref{eq.sol.Cauchy.Lip}. It includes the following stages:
{\begin{enumerate}
	\item[a)] computing the analytical continuation of  ${\mathcal F=({\mathcal W}_{\partial 
	\Omega_0} (u_{00}))^+ + ({\mathcal V} _{\partial \Omega_0} (u_{10})) ^+}$ from domain 
	$\Omega_0$ to domain  $\Omega_1 $.
	\item[b)] computing the function  $({\mathcal W}_{\partial \Omega_0} (u_{00}))^- +  
({\mathcal V}_{\partial \Omega_0} (u_{10}))^-$  in $\Omega_1 \setminus \Omega_0$; 
	\item[c)] computation of the Problem \ref{pr.Cauchy.M} solution $u$ in $\Omega_1 
	\setminus \Omega_0$  by formula \eqref{eq.sol.Cauchy.Lip}.
\end{enumerate}}
The idea of the proposed methods is to replace the analytical continuation with the solution 
to the Problem \ref{pr.ext.2} with an appropriate boundary datum. 

The boundary datum  of Problem \ref{pr.ext.2}  can be set on $\partial \Omega_0$. We denote 
it as $f_{\partial\Omega_0}$. The boundary datum can be  defined as:
$f_{\partial\Omega_0}(x)= \lim\limits_{y\to x}{\mathcal F}(y), x\in \partial\Omega_0, y\in 
\Omega_0$. According the definition of function ${\mathcal F}$, almost everywhere on 
$\partial\Omega_0$ 
\begin{equation} 
\label{f_Omega_0}
f_{\partial\Omega_0}(x) = -\frac{1}{2} u_{00} +{\rm v.p.}({\mathcal W}_{\partial \Omega_0} 
(u_{00}))^+ +
B_{0,\partial \Omega_0} ({\mathcal V} _{\partial \Omega_0} (u_{10})) ^+
\end{equation} 
where   ${\rm v.p.} {\mathcal W}_{\partial \Omega_0} (u_{00}) $ denotes the principal value 
of the double layer potential $({\mathcal W} ^{(B)}_{\partial \Omega_0} (u_{00})) (x)$ (see 
\eqref{eq.W.m=1}) at the point $x\in \partial \Omega_0$.

The application of the formula \eqref{f_Omega_0} requires the calculation of singular 
integrals. To avoid this, we can set a "virtual" surface $\partial\Omega_{\rm int}$, which 
bounds a relatively compact domain $\Omega_{\rm int}$ such that $\overline \Omega_{\rm int}
\subset \Omega_0$. We assume that $\partial\Omega_{\rm int}$  belongs to $C^{s}$-class of 
smoothness  if $s\in \mathbb N$, $s\geq 2$ 
or $\partial\Omega_{\rm int}$ is the Lipschitz surface if  $s \in [1, 3/2)$.

Instead of the analytical comtinuation of $\mathcal F$ from $\Omega_0$ to  $\Omega_1$ we can
perform the analytical continuation of  $\mathcal F$ from $\Omega_{\rm int}$ to  $\Omega_1$.  
Obviously, if the analytical continuation of $\mathcal F$ from $\Omega_0$ to  $\Omega_1$ 
exists, the analytical continuation  of  $\mathcal F$ from $\Omega_{\rm int}$ to  $\Omega_1$ 
also exists; according the Unicue Continuation property, it coinsides with the analytical 
continuation of $\mathcal F$ from  $\Omega_0$ to  $\Omega_1$.  

To perform that analitycal continuation we can solve Problem \ref{pr.ext.2} in the following 
statement: to find a function  $\mathcal F$ in $\Omega_1$ such that $L\mathcal F=0$ in 
$\Omega_1$ s.t. $B_{0,\partial\Omega_int}\mathcal F= f_{\partial\Omega_{\rm int}}$.
The boundary datum $f_{\partial\Omega_{\rm int}}$ is computed as:
\begin{equation} 
\label{f_Omega_int}
f_{\partial\Omega_{\rm int}}(x) = B_{0,\partial \Omega_{\rm int}}({\mathcal W}_{\partial 
\Omega_0} (u_{00}))^+ +
B_{0,\partial \Omega_{\rm int}} ({\mathcal V} _{\partial \Omega_0} (u_{10})) ^+
\end{equation} 

The reduction of the  Cauchy problem to Problem \ref{pr.ext.2} can also be performed in a 
slightly different form (cf. with \cite{Kalinin2019}). 

According formula  \eqref{eq.sol.Cauchy.Lip} the trace $\phi= B_{0,\partial \Omega_1}u$ on  $
\partial \Omega_1$ of the Cauchy problem solution $u$ is equal to:

$\phi= B_{0,\partial \Omega_1}({\mathcal W}_{\partial \Omega_0} (u_{00}))^- +  
B_{0,\partial \Omega_1}({\mathcal V}_{\partial \Omega_0} (u_{10}))^- - B_{0,\partial \Omega_1}
\mathcal F$ .

Let us consider the Dirichlet problem: $Lu=0$,  $u\in[H^{s} (\Omega_1)]^k$,  
$B_{0,{\partial\Omega_1}}u=\phi$, $\phi\in[H^{s-1/2} (\partial\Omega_1)]^k$, 
 and  intoduce an operator $D: \big[H^{s-1/2} (\partial \Omega_1)\big]^k \rightarrow 
\big[H^{s-1/2} (\partial \Omega_0)\big]^k$ which maps 
$\phi$ on  $\partial \Omega_1$ to $\hat f=B_{0,{\partial\Omega_0}}u$ on $\partial \Omega_0$.

Taking in account \eqref{f_Omega_0}, we can see that
\begin{equation} 
\label{f_hat}
\hat f= D\phi=D(B_{0,\partial \Omega_1}({\mathcal W}_{\partial \Omega_0} (u_{00}))^- +  
B_{0,\partial \Omega_1}({\mathcal V}_{\partial \Omega_0} (u_{10}))^-)-f_{\partial\Omega_0}.
\end{equation} 
Thus, the trace on $\partial\Omega_1$ of the solution to Problem \ref{pr.ext.2} with boundary 
data $\hat f$ coinsides with the trace on $\partial\Omega_1$ of the solution to the Cauchy 
problem. Therefore, the trace on $\partial\Omega_1$ of the solution to the Cauchy problem can 
be found by solving Problem \ref{pr.ext.2} with that boundary datum $\hat f$. 

Operator $D$ can be presented, for instance, as:
$D\phi= B_{0,{\partial\Omega_0}} {\mathcal V}_{\partial \Omega_1} ( B_{0,{\partial\Omega_1}}{
\mathcal V}_{\partial \Omega_1})^{-1}\phi$. Note, that  according to Corollary 
\ref{c.Dirichlet.Poisson} and Theorem \ref{t.single.layer},  
operator $B_{0,{\partial\Omega_1}}{\mathcal V}_{\partial \Omega_1}:\big[H^{s-3/2} (\partial 
\Omega_1)\big]^k \rightarrow \big[H^{s-1/2} (\partial \Omega_0)\big]^k $ 
is continuously invertible.

To justify the applicability of the reduction of the Cauchy problem to Problem \ref{pr.ext.2}, 
we note that transform $({\mathcal W}_{\partial \Omega_0} (u_{00}))^- +  
({\mathcal V}_{\partial \Omega_0} (u_{10}))^-$  from the space $\Big\{ [H^{s-1/2}(\partial 
\Omega_0)]^k, [H^{s-3/2}(\partial \Omega_0)]^k\Big\}$ to $[H^{s}(\Omega_1 \setminus \Omega_0)]
^k$ is continuous. Transforms  \eqref{f_Omega_0},\eqref{f_hat} and \eqref{f_Omega_int} also 
continuously map the Cauchy data $\{u_{00}, u_{10}\}$ from the space
 $\Big\{ [H^{s-1/2}(\partial \Omega_0)]^k, [H^{s-3/2}(\partial \Omega_0)]^k\Big\}$ to 
$f_{\partial\Omega_0}\in[H^{s-1/2}(\partial \Omega_0)]^k$ and $f_{\partial\Omega_{\rm int}}\in
[H^{s-1/2}(\partial \Omega_{\rm int})]^k$ respectively. 

To solve Problem \ref{pr.ext.2} (or approximate its solution) one can use the integral 
equations method of MFC. All the statements of Corollaries \ref{t.4.3} - \ref{t.4.6} hold for 
the boundary data of Problem \ref{pr.ext.2} which are obtained by \eqref{f_Omega_int}, 
\eqref{f_Omega_int} and \eqref{f_hat} from the Cauchy data  $\Big\{u_{00}\in[H^{s-1/2}(
\partial \Omega_0)]^k, u_{10}\in[H^{s-3/2}(\partial \Omega_0)]^k\Big\}.$

In conclusion, we would like to make a few short remarks about the numerical implementation 
of the methods  considered above. 

The indirect boundary integral equation method for solving Problem \ref{pr.ext.2} requires 
numerical solving operator equation \eqref{eq.4.10} and computing the problem solution in 
$\Omega_1$ via formula \eqref{eq.4.8} by numerical integration.

The conventional way of numerical solving linear operator equations consists of approximating 
them by systems of linear algebraic equations.  The boundary element method  (BEM) (in the 
collocation or Galerkin version) can be employed for this propose. Note that 
the results obtained in this paper for the second order elliptic operators include the case 
of a domain with Lipschitz boundaries. This justifies the use of the most popular version of 
BEM, which assumes a polygonal approximation (triangulation) of the boundaries. For the 
details of BEM implementation, see, for instance  \cite{rjasanow2007fast}. 

A matrix of sufficiently large dimension obtained by BEM approximation of operator of 
equation \eqref{eq.4.8} is expected to be ill-conditioned.  Therefore, to solve a system of 
linear algebraic equations approximating operator equation \eqref{eq.4.8}, it is necessary to 
apply regularization algorithms. Regularization algorithms of the Tikhonov type (which 
provide the boundedness of the respective norms of the solutions in finite-
dimensional spaces) can be used for this purpose.  Corollary \ref{t.4.4} justifies the 
applicability of this regularization approach.

To assemble the BEM matrix and calculate the problem solutions in $\Omega_1$ in a  
neighbourhood of $\partial\Omega_1$ by  \eqref{eq.4.8}, they have to compute weak singular 
and near weak singular integrals. MFS is free from this encumbrance. When using MFS to solve 
Problem \ref{pr.ext.2}, the unknown coefficients $c_i$ of the expansion \eqref{eq.4.11} of 
the problem solution by the system of fundamental solutions can be found by setting of  
collocation points $\{z_j\}_{j=1}^{N} \subset \partial \Omega_2$ and  $\{x_i\}_{i=1}^{N} 
\subset \partial \Omega_0$ and solving the following system of linear algebraic equations:

\begin{equation} 
\label{eq.mfs}
\sum_{j=1}^{N} 
\Phi(x_i,z_j) c_j ={u_{00}}(x_i)
\end{equation} 
where ${u_{00}}$ is the boundary datum of Problem \ref{pr.ext.2}.

Some rules for choosing the collocation points that ensure the uniqueness of solution to 
equation \eqref{eq.mfs} can be found in \cite{cheng2020overview}.

In general, the matrix of  equation \eqref{eq.mfs} is typically ill-conditioned and 
application of regularization methods for solving \eqref{eq.mfs} is required. Corollary 
\ref{t.4.6} justifies the applicability of the Tikhonov-type regularization.

The "extension"{} approach for approximation of the Problem \ref{pr.Dir} solution via 
numerical solving Problem \ref{pr.ext.2} by BEM or MFS leads to the ill-conditioned 
matrices.  Despite the fact that  solutions to the corresponding linear algebraic equation 
systems are very sensitive to  errors of the 
conditions, the final solutions to Problem \ref{pr.Dir} in $\Omega_0$ is 
stable. The statements of  Corollary \ref{t.4.4} and  Corollary \ref{t.4.6} explain this 
fact. However, too large values of the linear algebraic systems solutions can cause technical problems in the 
computer calculations. In this situation the double-precision arithmetic or using some regularization methods (see, for instance, \cite{MFS_Regularization}) can be necessary. 

Note that the conditionality of that BEM or MFS matrixes as well as the convergence rate of 
the approximation of the solution to Problem \ref{pr.Dir} depends on the geometry of domain  $
\Omega_0$ and  the embracing boundary $\partial \Omega_2$.  These regularities are still 
insufficiently investigated. Some observations on these issues (related to MFS) have been 
summarised in \cite{cheng2020overview}. 

The numerical solving of the Cauchy problem with the proposed methods includes a preliminary 
computation of the boundary datum for Problem \ref{pr.ext.2}. The computations by formula 
\eqref{f_hat} require inversion of operator $B_{0,{\partial\Omega_1}}{\mathcal V}_{\partial 
\Omega_1}$. The operator inversion can be computed by BEM. Corollary \ref{c.Dirichlet.Poisson}
 and Theorem \ref{t.single.layer} allow us to conclude that the BEM matrix which approximates 
this operator is well-conditioned for the common versions of BEM (see, for instance  
\cite{rjasanow2007fast} for more detailes). 

The implementation of formulas \eqref{f_Omega_0} and \eqref{f_hat} require computation of 
singular and weak singular integrals. Moreover, obtaining the solution component 
$({\mathcal W}_{\partial \Omega_0} (u_{00}))^- +  
({\mathcal V}_{\partial \Omega_0} (u_{10}))^-$  in $\Omega_1 \setminus \Omega_0$ in a  
neighbourhood of $\partial\Omega_0$ requires computing near singular and near weak singular 
integrals. The technique of numerical computation of the singular integrals can be found, for 
instance, in \cite{SingularIntegrals1998}, \cite{SingularIntegrals2015} and literature cited 
there.

%%%%%%%%%%%%%%%%%%%%%%%%%%%%%%%%%%%%%%%%%%%%%%%%%%%%%%%%%%%%%%

\smallskip

\textit{Acknowledgments\,} 
The second author was supported 
%The research on existence theorems 
%was supported 
by Sirius University of 
Science and Technology (project 'Spectral and Functional Inequalities of Mathematical 
Physics and Their Applications').
%The second author was supported by 
%the grant of the Foundation for the Advancement of Theoretical 
%Physics and Mathematics "BASIS"{}
%the Russian Science Foundation,  grant N 20-11-20117.

\end{document}